\documentclass[12pt,reqno]{amsart}
\usepackage{latexsym}
\usepackage{amssymb}
\usepackage{amscd}

\usepackage{comment}

\numberwithin{equation}{section}

\newtheorem{theorem}{Theorem}
\newtheorem{proposition}[theorem]{Proposition}

\newtheorem{corollary}[theorem]{Corollary}
\newtheorem{conjecture}[theorem]{Conjecture}

\theoremstyle{remark}
\newtheorem*{remark}{Remark}

\def\ord{\operatorname{ord}}
\def\SL{\operatorname{SL}}

\def\fl#1{\left\lfloor#1\right\rfloor}

\begin{document}

\title[Generating functions for sequences of squares]
{Infinite product formulae for generating functions for
sequences of squares}
\author[C.~Krattenthaler]{Christian Krattenthaler$^\dagger$}
\address[C.~Krattenthaler]{Universit{\"a}t Wien, Fakult{\"a}t f{\"u}r Mathematik, Oskar-Morgenstern-Platz 1, A-1090 Vienna, Austria. WWW: {\tt https://www.mat.univie.ac.at/\lower0.5ex\hbox{\~{}}kratt}.}
\author[M.~Merca]{Mircea Merca}
\address[M.~Merca]{Department of Mathematics, University of Craiova, Craiova, Romania; Academy of Romanian Scientists, Bucharest, Romania}
\author[C.-S.~Radu]{Cristian-Silviu Radu$^\dagger$}
\address[C.-S. Radu]{Research Institute for Symbolic Computation\\
J. Kepler University Linz,
A-4040 Linz, Austria. WWW: {\tt https://risc.jku.at/m/cristian-silviu-radu/}}

\thanks{$^\dagger$Supported by the Austrian
Science Foundation FWF (grant S50-N15)
in the framework of the Special Research Program
``Algorithmic and Enumerative Combinatorics".
}


\begin{abstract}
We state and prove product formulae for several generating functions
for sequences
$(a_n)_{n\ge0}$ that are defined by the property that $Pa_n+b^2$
is a square, where $P$ and $b$ are given integers. In particular,
we prove corresponding conjectures of the second author. We show
that, by means of the Jacobi triple product identity, 
all these generating functions can be reduced to a linear
combination of theta function products. The proof of our
formulae then consists in simplifying these linear combinations
of theta products into single products. We do this in two ways:
(1) by using modular function theory, and (2) by applying
the Weierstra{\ss} addition formula for theta products.
\end{abstract}

\maketitle

\section{Introduction}
Relations between infinite $q$-series and infinite $q$-products have
an honorable history starting with Euler and Gau{\ss} and first
studied systematically by Jacobi. Many identities of the form  
$$\text{infinite $q$-series = infinite $q$-product}$$
arise in number theory, analysis, combinatorics, the theory of integer
partitions, representation theory of Lie algebras, vertex 
operator algebras, knot theory, and 
statistical mechanics. Playing with series and products, Euler
discovered the pentagonal number theorem (cf.\ e.g.\ \cite[Ex.~2.18]{GaRaAA}), 
\begin{equation} \label{eq:pent} 
\sum_{n=-\infty}^\infty (-1)^n q^{n(3n-1)/2} = (q;q)_\infty,
\end{equation}
which is the very first theorem that is of this type.
Here, and in the following, $q$ is a complex number 
with\footnote{Alternatively, all definitions and identities may
  be understood in the sense of formal power series in~$q$.} $\vert
q\vert<1$,
and the symbol $(a;q)_\infty$ denotes the infinite product
$$
(a;q)_\infty:=
\prod _{i=0} ^{\infty}(1-aq^i).
$$
Moreover, we use the short notation
$$
(a_1,a_2,\dots,a_m;q):=
\prod _{j=1} ^{m}(a_j;q)_\infty.
$$

An alternative way to state Euler's identity is as follows:

\bigskip
{\it Let\/ $(a_n)_{n\ge0}=(0, 1, 2, 5, 7, 12, 15, 22, 26, 35,
40, \dots)$ be the sequence of non-negative integers $m$ such
that $24m+1$ is a square. Then
\begin{equation*}
\sum_{n=0}^\infty(-1)^{\fl{(n+1)/2}}q^{a_n}=
(q;q)_\infty.
\end{equation*}}
\bigskip

In \cite{MercAA}, the second author
studied --- among others --- series-product identities of this type and
listed several empirically found such formulae, see
\cite[Ids.~5.1--5.6, Ids.~6.1--6.14]{MercAA}.
A typical example (cf.\ \cite[Id.~6.1]{MercAA} and
Theorem~\ref{thm:240m+1} below) is the following
statement:

\bigskip
{\it Let\/ $(a_n)_{n\ge0}=(0, 4, 7, 10, 21, 26, 33, 59,
61, 95, 108, \dots)$ be the sequence of non-negative integers $m$ such
that $240m+1$ is a square. Then
\begin{equation} \label{eq:Id6.1A}
\sum_{n=0}^\infty(-1)^{\fl{(n+2)/4}}q^{a_n}=
 \frac {(q,q^7,q^8;q^8)_\infty\,(q^6,q^{10};q^{16})_\infty} 
{(q,q^4;q^5)_\infty}.
\end{equation}}
\bigskip

This paper started by the observation that, by the use of 
Jacobi's triple product identity, the proofs of all these conjectured
formulae can be reduced to the verification of certain identities between
(specialised) Jacobi theta functions. Now, it is a folklore
fact that, since these theta functions are modular functions
(for certain subgroups of $\SL_2(\mathbb Z)$), such identities are routinely
verifiable. This is what we did, first. Subsequently, however, we
wanted to have a conceptual understanding of the established formulae.
Moreover,
we were interested in whether there is more than just these, first empirically
found, formulae from \cite{MercAA}. Indeed, further computer
experiments led us to discover many more such formulae, but still
of sporadic nature. Nevertheless, altogether, they helped us
to come up with two parametric theorems that subsume several
of the earlier empirically found formulae under one roof.

Obviously, parametric theorems cannot be routinely verified anymore.
Rather we found that Weierstra\ss' three-term relation between
theta products is the ``magic" identity that is behind {\it all\/} of
our formulae. More precisely, (aside from four formulae that are
direct consequences of the Jacobi triple product identity) 
for one class of formulae the verification consists in a single
application of Weierstra\ss' relation which reduces the sum of two
theta products to a single theta product, while for a second class
of formulae the verification requires a double application of
Weierstra\ss' relation --- a first application
to reduce the sum of four theta products
to the sum of two theta products, and then another application to
reduce the latter to a single theta product. 

Our parametric
theorems belong both to the first class. At this point in time,
we are not able to offer a conceptual understanding for
our proofs for the second class of formulae in the sense that
we were not able to embed these into parametric families of formulae.
Rather, we performed some computer assisted searches for more
formulae in the second class that remained unsuccessful, suggesting
that there may not be parametric families in the second class but
only these sporadic formulae.

\medskip
Our paper is structured as follows.
In the next section we collect the empirically found identities
from \cite{MercAA}. (We have slightly altered the order in which they
are presented in order to provide a more systematic listing.) 
They become theorems by the proofs in
Sections~\ref{sec:mech} and~\ref{sec:Riemann}.
Section~\ref{sec:Ids2} lists the additional identities that we found 
by our computer experiments subsequent to the publication of
\cite{MercAA}. 

The first step in all our proofs is to apply
Jacobi's triple product identity that converts the left-hand sides
of our identities into a linear combination of products of
theta functions. We recall Jacobi's identity in
Section~\ref{sec:Jacobi}, where we also derive two (well-known)
corollaries that we need in some of the proofs.

There follows another section of preparatory character, namely
Section~\ref{sec:mod}, in which facts from the theory of modular
functions are collected that are relevant in our context.
Based on them, we explain in Section~\ref{sec:auto} how to routinely
verify identities between theta products. In the subsequent section,
Section~\ref{sec:mech}, we apply this methodology to prove all the
theorems from Sections~\ref{sec:Ids1} and~\ref{sec:Ids2}
(with the exception of the theorems that are
special cases of parametric families). 
We provide full details for Theorem~\ref{thm:840m+361}, while for
all other theorems proofs are given in a stenographic fashion
since the pattern is always the same. 

Section~\ref{sec:theta} is again preparatory. There, we recall
the earlier mentioned Weierstra\ss' addition formula, and two
of its corollaries that we shall use particularly frequently.
Then, in Section~\ref{sec:Riemann}, we
present our proofs of all theorems from Sections~\ref{sec:Ids1}
and~\ref{sec:Ids2} (again with the exception of the theorems that are
special cases of parametric families) 
by the use of the Weierstra{\ss} relation.
Again, we give full details for the proof of Theorem~1, while for
all other theorems proofs are presented only in an abridged fashion.

The contents of Section~\ref{sec:Ids3} are
two parametric families of formulae of the type as in
\eqref{eq:Id6.1A}; one of them consists in Theorem~\ref{thm:24Pm+a^2}
and Corollary~\ref{cor:24Pm+a^2} (although these contain different
statements, their proofs are the same), and the other in
Theorem~\ref{thm:3Pm+a^2}
and Corollary~\ref{cor:3Pm+a^2} (again, these contain different
statements, but their proofs are the same).
Finally, in Theorem~\ref{thm:16m+a^2} we unify the statements
of Theorems~\ref{thm:16m+1} and~\ref{thm:16m+9}, and we provide
a uniform proof. 

We close the article by mentioning some consequences
and open problems in Section~\ref{sec:open}.

\section{Generating functions for sequences of squares}
\label{sec:Ids1}

Here, we list the empirically found formulae for generating functions
for sequences of squares from \cite{MercAA}. They become theorems
by the proofs in Sections~\ref{sec:mech} and~\ref{sec:Riemann},
respectively. 

\begin{theorem}[conjectured in {\cite[Id.~5.1]{MercAA}}]
\label{thm:840m+361}
Let\/ $(a_n)_{n\ge0}$ be the sequence of non-negative integers $m$ such
that $840m+361$ is a square. Then
\begin{equation} \label{eq:Id5.1}
\sum_{n=0}^\infty(-1)^{t(n)}q^{a_n}=
 \frac {(q,q^6,q^7;q^7)_\infty} {(q,q^4;q^5)_\infty},
\end{equation}
where
$$
t(n)=\begin{cases} 
0,&\text{if }n\equiv0,1,3,5,10,12,14,15\ (\text{\em mod }16),\\
1,&\text{otherwise.}
\end{cases}
$$
\end{theorem}

\begin{theorem}[conjectured in {\cite[Id.~5.2]{MercAA}}]
\label{thm:840m+529}
Let\/ $(a_n)_{n\ge0}$ be the sequence of non-negative integers $m$ such
that $840m+529$ is a square. Then
\begin{equation} \label{eq:Id5.2}
\sum_{n=0}^\infty(-1)^{t(n)}q^{a_n}=
 \frac {(q,q^6,q^7;q^7)_\infty} {(q^2,q^3;q^5)_\infty},
\end{equation}
where
$$
t(n)=\begin{cases} 
0,&\text{if }n\equiv0,2,3,6,9,12,13,15\ (\text{\em mod }16),\\
1,&\text{otherwise.}
\end{cases}
$$
\end{theorem}

\begin{theorem}[conjectured in {\cite[Id.~5.3, corrected]{MercAA}}]
\label{thm:840m+121}
Let\/ $(a_n)_{n\ge0}$ be the sequence of non-negative integers $m$ such
that $840m+121$ is a square. Then
\begin{equation} \label{eq:Id5.3}
\sum_{n=0}^\infty(-1)^{\fl{(n+4)/8}}q^{a_n}=
 \frac {(q^2,q^5,q^7;q^7)_\infty} {(q,q^4;q^5)_\infty}.
\end{equation}
\end{theorem}

\begin{theorem}[conjectured in {\cite[Id.~5.4]{MercAA}}]
\label{thm:840m+289}
Let\/ $(a_n)_{n\ge0}$ be the sequence of non-negative integers $m$ such
that $840m+289$ is a square. Then
\begin{equation} \label{eq:Id5.4}
\sum_{n=0}^\infty(-1)^{t(n)}q^{a_n}=
 \frac {(q^2,q^5,q^7;q^7)_\infty} {(q^2,q^3;q^5)_\infty},
\end{equation}
where
$$
t(n)=\begin{cases} 
0,&\text{if }n\equiv0,1,3,5,10,12,14,15\ (\text{\em mod }16),\\
1,&\text{otherwise.}
\end{cases}
$$
\end{theorem}

\begin{theorem}[conjectured in {\cite[Id.~5.5, corrected]{MercAA}}]
\label{thm:840m+1}
Let\/ $(a_n)_{n\ge0}$ be the sequence of non-negative integers $m$ such
that $840m+1$ is a square. Then
\begin{equation} \label{eq:Id5.5}
\sum_{n=0}^\infty(-1)^{\fl{(n+4)/8}}q^{a_n}=
 \frac {(q^3,q^4,q^7;q^7)_\infty} {(q,q^4;q^5)_\infty}.
\end{equation}
\end{theorem}

\begin{theorem}[conjectured in {\cite[Id.~5.6]{MercAA}}]
\label{thm:840m+169}
Let\/ $(a_n)_{n\ge0}$ be the sequence of non-negative integers $m$ such
that $840m+169$ is a square. Then
\begin{equation} \label{eq:Id5.6}
\sum_{n=0}^\infty(-1)^{t(n)}q^{a_n}=
 \frac {(q^3,q^4,q^7;q^7)_\infty} {(q^2,q^3;q^5)_\infty},
\end{equation}
where
$$
t(n)=\begin{cases} 
0,&\text{if }n\equiv0,1,2,4,11,13,14,15\ (\text{\em mod }16),\\
1,&\text{otherwise.}
\end{cases}
$$
\end{theorem}

\begin{theorem}[conjectured in {\cite[Id.~6.1, corrected]{MercAA}}]
\label{thm:240m+1}
Let\/ $(a_n)_{n\ge0}$ be the sequence of non-negative integers $m$ such
that $240m+1$ is a square. Then
\begin{equation} \label{eq:Id6.1}
\sum_{n=0}^\infty(-1)^{\fl{(n+2)/4}}q^{a_n}=
 \frac {(q,q^7,q^8;q^8)_\infty\,(q^6,q^{10};q^{16})_\infty} 
{(q,q^4;q^5)_\infty}.
\end{equation}
\end{theorem}

\begin{theorem}[conjectured in {\cite[Id.~6.2]{MercAA}}]
\label{thm:240m+49}
Let\/ $(a_n)_{n\ge0}$ be the sequence of non-negative integers $m$ such
that $240m+49$ is a square. Then
\begin{equation} \label{eq:Id6.2}
\sum_{n=0}^\infty(-1)^{\fl{5n/4}}q^{a_n}=
 \frac {(q,q^7,q^8;q^8)_\infty\,(q^6,q^{10};q^{16})_\infty} 
{(q^2,q^3;q^5)_\infty}.
\end{equation}
\end{theorem}

\begin{theorem}[conjectured in {\cite[Id.~6.5, corrected]{MercAA}}]
\label{thm:240m+121}
Let\/ $(a_n)_{n\ge0}$ be the sequence of non-negative integers $m$ such
that $240m+121$ is a square. Then
\begin{equation} \label{eq:Id6.5}
\sum_{n=0}^\infty(-1)^{\fl{(n+2)/4}}q^{a_n}=
 \frac {(q^3,q^5,q^8;q^8)_\infty\,(q^2,q^{14};q^{16})_\infty} 
{(q,q^4;q^5)_\infty}.
\end{equation}
\end{theorem}

\begin{theorem}[conjectured in {\cite[Id.~6.6]{MercAA}}]
\label{thm:240m+169}
Let\/ $(a_n)_{n\ge0}$ be the sequence of non-negative integers $m$ such
that $240m+169$ is a square. Then
\begin{equation} \label{eq:Id6.6}
\sum_{n=0}^\infty(-1)^{\fl{5n/4}}q^{a_n}=
 \frac {(q^3,q^5,q^8;q^8)_\infty\,(q^2,q^{14};q^{16})_\infty} 
{(q^2,q^3;q^5)_\infty}.
\end{equation}
\end{theorem}

\begin{theorem}[conjectured in {\cite[Id.~6.3, corrected]{MercAA}}]
\label{thm:15m+1}
Let\/ $(a_n)_{n\ge0}$ be the sequence of non-negative integers $m$ such
that $15m+1$ is a square. Then
\begin{equation} \label{eq:Id6.3}
\sum_{n=0}^\infty(-1)^{\fl{(n+2)/4}}q^{a_n}=
 \frac {(q^2,q^6,q^8;q^8)_\infty\,(q^4,q^{12};q^{16})_\infty} 
{(q,q^4;q^5)_\infty}.
\end{equation}
\end{theorem}

\begin{theorem}[conjectured in {\cite[Id.~6.4, corrected]{MercAA}}]
\label{thm:15m+4}
Let\/ $(a_n)_{n\ge0}$ be the sequence of non-negative integers $m$ such
that $15m+4$ is a square. Then
\begin{equation} \label{eq:Id6.4}
\sum_{n=0}^\infty(-1)^{\fl{(n+2)/4}}q^{a_n}=
 \frac {(q^2,q^6,q^8;q^8)_\infty\,(q^4,q^{12};q^{16})_\infty} 
{(q^2,q^3;q^5)_\infty}.
\end{equation}
\end{theorem}

\begin{theorem}[conjectured in {\cite[Id.~6.7]{MercAA}}]
\label{thm:6.7}
We have
\begin{equation} \label{eq:Id6.7}
  \sum_{n=-\infty}^\infty q^{n(5n+1)}
= \frac {(q,q^9,q^{10};q^{10})_\infty\,(q^8,q^{12};q^{20})_\infty}
{(q,q^4;q^5)_\infty}.
\end{equation}
\end{theorem}

\begin{theorem}[conjectured in {\cite[Id.~6.8]{MercAA}}]
\label{thm:6.8}
We have
\begin{equation} \label{eq:Id6.8}
  \sum_{n=0}^\infty \big(q^{n(n+1)}- q^{5n(n+1)+1}\big)
= \frac {(q,q^9,q^{10};q^{10})_\infty\,(q^8,q^{12};q^{20})_\infty}
{(q^2,q^3;q^5)_\infty}.
\end{equation}
\end{theorem}

\begin{theorem}[conjectured in {\cite[Id.~6.9]{MercAA}}]
\label{thm:6.9}
We have
\begin{equation} \label{eq:Id6.9}
1+  \sum_{n=1}^\infty \big(q^{n^2}+q^{5n^2}\big)
= \frac {(q^2,q^8,q^{10};q^{10})_\infty\,(q^6,q^{14};q^{20})_\infty}
{(q,q^4;q^5)_\infty}.
\end{equation}
\end{theorem}

\begin{theorem}[conjectured in {\cite[Id.~6.10]{MercAA}}]
\label{thm:6.10}
We have
\begin{equation} \label{eq:Id6.10}
  \sum_{n=-\infty}^\infty q^{n(5n+2)}
= \frac {(q^2,q^8,q^{10};q^{10})_\infty\,(q^6,q^{14};q^{20})_\infty}
{(q^2,q^3;q^5)_\infty}.
\end{equation}
\end{theorem}

\begin{theorem}[conjectured in {\cite[Id.~6.11]{MercAA}}]
\label{thm:6.11}
We have
\begin{equation} \label{eq:Id6.11}
  \sum_{n=0}^\infty \big(q^{n(n+1)}+ q^{5n(n+1)+1}\big)
= \frac {(q^3,q^7,q^{10};q^{10})_\infty\,(q^4,q^{16};q^{20})_\infty}
{(q,q^4;q^5)_\infty}.
\end{equation}
\end{theorem}

\begin{theorem}[conjectured in {\cite[Id.~6.12]{MercAA}}]
\label{thm:6.12}
We have
\begin{equation} \label{eq:Id6.12}
  \sum_{n=-\infty}^\infty q^{n(5n+3)}
= \frac {(q^3,q^7,q^{10};q^{10})_\infty\,(q^4,q^{16};q^{20})_\infty}
{(q^2,q^3;q^5)_\infty}.
\end{equation}
\end{theorem}

\begin{theorem}[conjectured in {\cite[Id.~6.13]{MercAA}}]
\label{thm:6.13}
We have
\begin{equation} \label{eq:Id6.13}
  \sum_{n=-\infty}^\infty q^{n(5n+4)}
= \frac {(q^4,q^6,q^{10};q^{10})_\infty\,(q^2,q^{18};q^{20})_\infty}
{(q,q^4;q^5)_\infty}.
\end{equation}
\end{theorem}

\begin{theorem}[conjectured in {\cite[Id.~6.14]{MercAA}}]
\label{thm:6.14}
We have
\begin{equation} \label{eq:Id6.14}
  \sum_{n=1}^\infty \big(q^{n^2-1}- q^{5n^2-1}\big)
= \frac {(q^4,q^6,q^{10};q^{10})_\infty\,(q^2,q^{18};q^{20})_\infty}
{(q^2,q^3;q^5)_\infty}.
\end{equation}
\end{theorem}

\section{More generating functions for sequences of squares}
\label{sec:Ids2}

In this section, we collect the additional formulae for generating
functions for sequences of squares that we found when we started
our work on this kind of identities. Also these become theorems
by the proofs in Sections~\ref{sec:mech} and~\ref{sec:Riemann},
respectively. 

\begin{theorem} 
\label{thm:120m+49}
	Let\/ $(a_n)_{n\ge0}$ be the sequence of non-negative integers $m$ such
	that $120m+49$ is a square. Then
	\begin{equation} \label{eq:1}
	\sum_{n=0}^\infty(-1)^{\fl{5n/4}}q^{a_n}=
	\frac {(q,q^4,q^5;q^5)_\infty} {(q^2,q^3;q^5)_\infty}.
	\end{equation}
\end{theorem}

\begin{theorem} 
\label{thm:120m+1}
	Let\/ $(a_n)_{n\ge0}$ be the sequence of non-negative integers $m$ such
	that $120m+1$ is a square. Then
	\begin{equation} \label{eq:2}
	\sum_{n=0}^\infty(-1)^{\fl{(n+2)/4}}q^{a_n}=
	\frac {(q^2,q^3,q^5;q^5)_\infty} {(q,q^4;q^5)_\infty} .
	\end{equation}
\end{theorem}

\begin{theorem} \label{thm:168m+121}
	Let\/ $(a_n)_{n\ge0}$ be the sequence of non-negative integers $m$ such
	that $168m+121$ is a square. Then
	\begin{equation} \label{eq:3}
	\sum_{n=0}^\infty(-1)^{\fl{5n/4}}q^{a_n}=
	\frac {(q,q^6,q^7;q^7)_\infty} {(q^3,q^4;q^7)_\infty}.
	\end{equation}
\end{theorem}

\begin{theorem} \label{thm:168m+1}
	Let\/ $(a_n)_{n\ge0}$ be the sequence of non-negative integers $m$ such
	that $168m+1$ is a square. Then
	\begin{equation} \label{eq:4}
	\sum_{n=0}^\infty(-1)^{\fl{(n+2)/4}}q^{a_n}=
	\frac {(q^2,q^5,q^7;q^7)_\infty} {(q,q^6;q^7)_\infty}.
	\end{equation}
\end{theorem}

\begin{theorem} \label{thm:168m+25}
	Let\/ $(a_n)_{n\ge0}$ be the sequence of non-negative integers $m$ such
	that $168m+25$ is a square. Then
	\begin{equation} \label{eq:5}
	\sum_{n=0}^\infty(-1)^{\fl{(n+2)/4}}q^{a_n}=
	\frac {(q^3,q^4,q^7;q^7)_\infty} {(q^2,q^5;q^7)_\infty}.
	\end{equation}
\end{theorem}

\begin{theorem} \label{thm:48m+1}
	Let\/ $(a_n)_{n\ge0}$ be the sequence of non-negative integers $m$ such
	that $48m+1$ is a square. Then
	\begin{equation} \label{eq:6}
	\sum_{n=0}^\infty(-1)^{\fl{(n+2)/4}}q^{a_n}=
	\frac {(q^2,q^6,q^8;q^8)_\infty} {(q,q^7;q^8)_\infty}.
	\end{equation}
\end{theorem}

\begin{theorem} 
\label{thm:48m+25}
	Let\/ $(a_n)_{n\ge0}$ be the sequence of non-negative integers $m$ such
	that $48m+25$ is a square. Then
	\begin{equation} \label{eq:7}
	\sum_{n=0}^\infty(-1)^{\fl{5n/4}}q^{a_n}=
	\frac {(q^2,q^6,q^8;q^8)_\infty} {(q^3,q^5;q^8)_\infty}.
	\end{equation}
\end{theorem}

\begin{theorem} 
\label{thm:21m+1}
	Let\/ $(a_n)_{n\ge0}$ be the sequence of non-negative integers $m$ such
	that $21m+1$ is a square. Then
	\begin{equation} \label{eq:8}
	\sum_{n=0}^\infty(-1)^{\fl{(n+2)/4}}q^{a_n}=
	\frac {(q,q^6,q^7;q^7)_\infty (q^5,q^9;q^{14})_\infty} {(q,q^3;q^4)_\infty}.
	\end{equation}
\end{theorem}

\begin{theorem} \label{thm:21m+4}
	Let\/ $(a_n)_{n\ge0}$ be the sequence of non-negative integers $m$ such
	that $21m+4$ is a square. Then
	\begin{equation} \label{eq:9}
	\sum_{n=0}^\infty(-1)^{\fl{(n+2)/4}}q^{a_n}=
	\frac {(q^2,q^5,q^7;q^7)_\infty (q^3,q^{11};q^{14})_\infty} {(q,q^3;q^4)_\infty}.
	\end{equation}
\end{theorem}

\begin{theorem} \label{thm:21m+16}
	Let\/ $(a_n)_{n\ge0}$ be the sequence of non-negative integers $m$ such
	that $21m+16$ is a square. Then
	\begin{equation} \label{eq:10}
	\sum_{n=0}^\infty(-1)^{\fl{5n/4}}q^{a_n}=
	\frac {(q^3,q^4,q^7;q^7)_\infty (q,q^{13};q^{14})_\infty} {(q,q^3;q^4)_\infty}.
	\end{equation}
\end{theorem}

\begin{theorem} \label{thm:16m+1}
	Let\/ $(a_n)_{n\ge0}$ be the sequence of non-negative integers $m$ such
	that $16m+1$ is a square. Then
	\begin{equation} \label{eq:11}
	\sum_{n=0}^\infty q^{a_n}=
	\frac {(q,q^7,q^8;q^8)_\infty (q^6,q^{10};q^{16})_\infty} {(q,q^3;q^4)_\infty}.
	\end{equation}
\end{theorem}

\begin{theorem} \label{thm:16m+9}
	Let\/ $(a_n)_{n\ge0}$ be the sequence of non-negative integers $m$ such
	that $16m+9$ is a square. Then
	\begin{equation} \label{eq:12}
	\sum_{n=0}^\infty q^{a_n}=
	\frac {(q^3,q^5,q^8;q^8)_\infty (q^2,q^{14};q^{16})_\infty} {(q,q^3;q^4)_\infty}.
	\end{equation}
\end{theorem}

\section{The Jacobi triple product identity and two of its
  consequences}
\label{sec:Jacobi}

The purpose of this section is, first of all, to state the Jacobi
triple product identity which is ubiquitously used in the proofs
of our theorems, and, second, to make two corollaries explicit
that we need in the proofs of four of our theorems.

\medskip
The Jacobi triple product identity says that
(cf.\ \cite[Eq.~(1.6.1)]{GaRaAA})
\begin{equation} \label{eq:JTP} 
\sum_{n=-\infty}^\infty (-1)^n q^{\binom n2}z^n
=
(q,z,q/z;q)_\infty.
\end{equation}
Letting $q\to q^2$ and setting $z=-q$ in \eqref{eq:JTP}, we obtain
$$
\sum_{n=-\infty}^\infty q^{n^2}
=
(q^2,-q,-q;q^2)_\infty.
$$
Since we have
$$
\sum_{n=1}^\infty q^{n^2}
=
\sum_{n=-\infty}^{-1} q^{n^2},
$$
the previous identity implies
\begin{equation} \label{eq:n^2} 
\sum_{n=1}^\infty q^{n^2}
=
\frac {1} {2}\big((q^2,-q,-q;q^2)_\infty-1\big).
\end{equation}
We shall need this identity in the
proofs of Theorems~\ref{thm:6.9} and \ref{thm:6.14}.

\medskip
On the other hand,
letting $q\to q^2$ and setting $z=-q^2$ in \eqref{eq:JTP}, we obtain
$$
\sum_{n=-\infty}^\infty q^{n^2+n}
=
2\,(q^2,-q^2,-q^2;q^2)_\infty.
$$
Since we have
$$
\sum_{n=0}^\infty q^{n^2+n}
=
\sum_{n=-\infty}^{-1} q^{n^2+n},
$$
the previous identity implies
\begin{equation} \label{eq:n(n+1)} 
\sum_{n=0}^\infty q^{n^2+n}
=
(q^2,-q^2,-q^2;q^2)_\infty.
\end{equation}
This identity will be used in the
proofs of Theorems~\ref{thm:6.8} and \ref{thm:6.11}.

\section{Background on modular functions}
\label{sec:mod}

In this section, we give a brief introduction to modular functions,
tailored to our purposes. Let
$\mathbb{H}:=\{x\in\mathbb{C}:\mbox{Im}(x)>0\}$ denote the upper half
plane. Roughly speaking, modular functions are (certain) 
meromorphic functions on
$\mathbb{H}$ that are invariant under the action of
a subgroup $\Gamma$ of $\SL_2(\mathbb{Z})$.
In our setting, $\Gamma=\Gamma_1(N)$ and $N>3$, where 
$$
\Gamma_1(N):=\left\{\begin{pmatrix} a&b\\c&d\end{pmatrix}\in
  \SL_2(\mathbb Z):a,d\equiv1~(\text{mod }N)\text{ and
  }c\equiv0~(\text{mod }N)\right\}.
$$

The crucial fact on which we base the ``methodology" explained in
the next section, is
the following proposition (cf.\ \cite[Prop.~4.12]{MiraAA}).
\begin{proposition}\label{mainprop}
  Let\/ $f$ be a non-constant meromorphic function on a compact Riemann
  surface $X$. Then 
  $$\sum_{p\in X}\ord_p(f)=0,$$
where $\ord_p(f)$ is the order of the Laurent series expansion of
$f$ about the point~$p$.
  \end{proposition}

For $\gamma=\left(\begin{smallmatrix} a & b \\ c &
  d \end{smallmatrix}\right)$ and $\tau\in\mathbb{H}$ we define
$\gamma\tau:=\frac{a\tau+b}{c\tau+d}$.
For a subgroup $\Gamma$ of $\SL_2(\mathbb Z)$, we let
$$A(\Gamma):=\{ \text{$f$ meromorphic on $\mathbb{H}$}:\text{
  $f(\gamma\tau)=f(\tau)$ for all $\gamma\in \Gamma$ and
  $\tau\in\mathbb{H}$}\}.$$ 


A function $f\in A(\Gamma_1(N))$ can be viewed as a meromorphic
function $\tilde{f}$ on the Riemann surface
$\mathbb{H}/\Gamma_1(N):=\{[\tau]:\tau\in\mathbb{H}\}$.
Here, $[\tau]:=\{\gamma\tau:\gamma\in\Gamma_1(N)\}$ denotes the orbit
of $\tau\in\mathbb{H}$ under the action of the group $\Gamma_1(N)$, and
$\tilde{f}([\tau]):=f(\tau)$. One equips $\mathbb{H}/\Gamma_1(N)$
with the quotient topology so that it is a topological space. To make
$\mathbb{H}/\Gamma_1(N)$ into a Riemann surface, we follow the recipe in
\cite[Sec.~1.8]{MiyaAA} (cf.\ also \cite{PaRaAA}). Then $\tilde{f}$ becomes
a meromorphic function on $\mathbb{H}/\Gamma_1(N)$. However, as we
shall explain in Section~\ref{sec:auto}, 
we want to use Proposition~\ref{mainprop}, which
is an assertion on meromorphic functions on a {\it compact} Riemann
surface. In other words,
we would need $\mathbb{H}/\Gamma_1(N)$ to
be compact, which it is not. So we need to add some points in order to
make it compact. In order to achieve this, we define
$\mathbb{H}^*:=\mathbb{H}\cup \mathbb{Q}\cup \{i\infty\}$. We extend
the action of 
$\gamma=\left(\begin{smallmatrix} a & b \\ c & d \end{smallmatrix}\right)$
to $\mathbb H^*$ as follows.
For $\frac{s}{t}\in\mathbb{Q}$, we let
$\gamma\frac{s}{t}:=\frac{a(s/t)+b}{c(s/t)+d}$ if $c(s/t)+d\neq 0$ and
$\gamma\frac{s}{t}:=i\infty$ otherwise. 
Moreover, we let $\gamma(i\infty) =\frac{a}{c}$ if
$c\neq 0$ and $\gamma(i\infty):=i\infty$ otherwise. Then
$\mathbb{H}^*/\Gamma_1(N):=\{[\tau]:\tau\in\mathbb{H}^*\}$ is a
compact topological space when equipped with the quotient
topology. More precisely, 
the topology of $\mathbb{H}^*$ is generated by the topology
of $\mathbb{H}$ (which inherits the standard topology in $\mathbb{C}$)
and the sets $U_M:=\{x\in\mathbb{H}:\mbox{Im}(x)>M\}\cup \{i\infty\}$,
and the sets  $\gamma U_M$ for all $\gamma\in
\SL_2(\mathbb{Z})$, where $M$ is a positive real number.   

In particular $X_1(N):=\mathbb{H}^*/\Gamma_1(N)$ is made into a
Riemann surface following the recipe given in \cite[Sec.~1.8]{MiyaAA}
(cf.\ also \cite{PaRaAA}). 
Again, this creates a problem:
for a function $f\in A(\Gamma_1(N)$, the corresponding 
function $\tilde{f}$ on the quotient space is not necessarily
meromorphic on $X_1(N)$. The problematic points are the cusps
$\{[s/t]:s/t\in \mathbb{Q}\cup \{i\infty\}\}$. Given a 
function $f\in A(\Gamma_1(N))$, in order for $\tilde{f}$
to be meromorphic on $X_1(N)$, we need that,
for each reduced fraction 
$\frac{a}{c}\in\mathbb{Q}$,  $f$ can be expressed as a
Laurent series with finite principal part in powers of $e^{\frac{2\pi
    i\gamma^{-1}\tau}{h_{c}}}$. Here, $h_{c}:=N/\gcd(c,N)$ is called
the width of the cusps $a/c$.
When these additional conditions are satisfied we say that
$f$ is a modular function for the group $\Gamma_1(N)$, and we write
$f\in M(\Gamma_1(N))$. In particular, $f\in M(\Gamma_1(N))$ implies
that $\tilde{f}$ is meromorphic on $X_1(N)$. 
Furthermore,
$\ord_{[a/c]}\tilde{f}$ equals the order of the Laurent series of
$f$ in powers of $e^{\frac{2\pi i\gamma^{-1}\tau}{h_{c}}}$. 
For an arbitrary function $f\in M(\Gamma_1(N))$, the order of
$\tilde{f}$ at a point $[\tau_0]$ for $\tau_0\in\mathbb{H}$ is $t$,
where $t$ is the order of $f$ when expanded in powers of
$\tau-\tau_0$.

\section{How to ``mechanically" prove theta function identities}
\label{sec:auto}

Here we outline the proof strategy for identities equating sums
of theta products that we are going to apply in the proofs of
our theorems from Sections~\ref{sec:Ids1} and~\ref{sec:Ids2}
given in the next section.

\medskip
Let us fix a positive modulus $N>3$.
For $g\in\{0,\dots,N-1\}$ define
\begin{equation} \label{eq:E} 
E_g=E_g(q;N):=q^{NB_2(g/N)/2}(q^g,q^{N-g};q^N)_\infty,
\end{equation}
where $B_2(x)=x^2-x+\frac {1} {6}$ is the second Bernoulli polynomial.
The function $E_g$ is essentially a (specialised) Jacobi theta
function, ``essentially" referring to the missing factor
$(q^N;q^N)_\infty$.
In the following we explain how identities of the form
\begin{equation} \label{eq:Eid} 
\sum_{j=1}^rc_j\prod_{g}E_g^{a_g^{(j)}}=0,
\end{equation}
where the $c_j$'s are complex numbers and the $a_g^{(j)}$'s are
integers, can be routinely verified if each summand in the sum above
is a modular function for the group $\Gamma_1(N)$.
It should be noted that the
left-hand side of \eqref{eq:Eid} is a linear combination of theta
products. 

Now, to verify the identity \eqref{eq:Eid}, the following steps have
to be performed. For convenience, in the following we write LHS for
the left-hand side of \eqref{eq:Eid}.

\medskip
{\sc Step 1.}
According to \cite[Prop.~3]{YY}, $\prod_{g}E_g^{a_g}(q;N)$ is a modular
function in $\tau$ for the group $\Gamma_1(N)$, where $q=e^{2\pi i\tau}$,
that is, $\prod_{g}E_g^{a_g}$ is an element of
$M(\Gamma_1(N))$ if 
\begin{equation} \label{eq:Ecrit} 
\sum_{g}a_g\equiv 0\pmod{12} \quad \text{ and } \quad 
\sum_g g^2a_g\equiv 0 \pmod{y(N)},
\end{equation}
where $y(N)=2N$ if $N$ is even, and $y(N)=N$ if $N$ is odd.
We use this criterion in LHS
for each summand in order to check that each summand is a modular
function for $\Gamma_1(N)$.

\medskip
{\sc Step 2.} Representatives of the cusps 
for the group $\Gamma_1(N)$ are computed. (There are only finitely many.)
The computer algebra programme {\sl Magma} provides an implementation
in form of the function \texttt{Cusps(Gamma1($N$))}. 

\medskip
{\sc Step 3.} For each (representative of a) cusp --- except~$i\infty$, 
$c$ say, and $j=1,2,\dots,r$, the order
of the function $\prod_{g}E_g^{a_g^{(j)}}$ at~$c$ has to be computed.
According to \cite[Prop.~4]{YY}, the order $\ord(E_g;c,N)$ of the
function $E_g$ at the cusp~$c$ of the group $\Gamma_1(N)$ is given
by 
\begin{equation} \label{eq:ord} 
\ord(E_g;c,N)=\frac {1} {2}\gcd(D_c,N)B_2(\{N_cg/\gcd(D_c,N)\}),
\end{equation}
where $D_c$ is the denominator of~$c$ and $N_c$ is the numerator
of~$c$, while $\{\alpha\}$ denotes the fractional part of the rational
number~$\alpha$.
The order of $\prod_{g}E_g^{a_g^{(j)}}$ at~$c$ then is
$$
\sum_g a_g^{(j)}\ord(E_g;c,N).
$$

\medskip
{\sc Step 4.} We obtain a lower bound on the order of LHS at $c$
by taking the minimum of the orders of the
individual summands of the sum in LHS.

\medskip
{\sc Step 5.} By Proposition~\ref{mainprop}, if $LHS$ is not
identically zero, then
the sum of all the orders equals zero. Hence, (again assuming
that $LHS$ is not identically zero)
for the function
$\widetilde{LHS}$ on the compact Riemann surface $X_1(N)$, we have
$$
0=\ord_{[i\infty]}(\widetilde{LHS})
+\sum_{\text{$[c]$, $c$ a cusp, $[c]\ne[i\infty]$}}\ord_{[c]}(\widetilde{LHS})
+\sum_{\text{$[p]$, $p$ not a cusp}}\ord_{[p]}(\widetilde{LHS}).
$$
If we sum all the lower bounds that we found in Step~4 over all the
cusps different from~$i\infty$, then we obtain a lower bound, $-U$ say, on the
first sum in the above expression. Furthermore, from the definition
of the function $E_g$ it is obvious that it cannot have a singularity
at a point $p\in\mathbb H$, and thus the order of $E_g$ at
$p$ is non-negative. This implies directly
that the orders $\ord_{[p]}(\widetilde{LHS})$ are non-negative, yielding
the lower bound $0$ on the second sum. Everything combined, we see
that $\ord_{[i\infty]}(\widetilde{LHS})\le U$.

\medskip
{\sc Step~6}. We now verify by direct computation that LHS
has the power series expansion
$0+0q+\dots+0q^{U}+\cdots$. This says that
$\ord_{[i\infty]}(\widetilde{LHS})>U$ (the reader should recall that,
under the relation $q=e^{2\pi i\tau}$, the point $\tau=i\infty$
corresponds to $q=0$), a contradiction to our
finding in Step~5 under the assumption that $LHS$ is not
identically zero.
Consequently, LHS must be the zero function. 

\section{``Mechanical" proofs}
\label{sec:mech}

This section is devoted to the presentation of the proofs of the
theorems in Sections~\ref{sec:Ids1} and~\ref{sec:Ids2} that are
based on the procedure outlined in the previous section.
We provide full details for the proof of Theorem~\ref{thm:840m+361},
while we remain brief for the proofs of the other theorems, all
of them being completely analogous. For the theorems which are
specialisations of the parametric theorems in Section~\ref{sec:Ids3},
we refer to the proofs given there.

\begin{proof}[Proof of Theorem \ref{thm:840m+361}]

{\sc Step 0.} We write the sum on the left-hand side
of \eqref{eq:Id5.1} explicitly, and then apply the Jacobi triple
product identity to obtain an expression that is a linear combination
of products of theta functions.

In order to accomplish this, we first observe 
that squares that are congruent to $361$
modulo $840$ are of the form $S^2$, where
$S\equiv19,61,79,89,121,131,149,191$~$(\text{mod}~210)$.
Consequently, taking the definition of $t(n)$ into account, we have
\begin{multline} \label{eq:Id5.1A}
\sum_{n=0}^\infty(-1)^{t(n)}q^{a_n}=
\sum_{k=0}^\infty (-1)^kq^{\frac {1} {840}((210k+19)^2-361)}
+\sum_{k=0}^\infty (-1)^kq^{\frac {1} {840}((210k+61)^2-361)}\\
-\sum_{k=0}^\infty (-1)^kq^{\frac {1} {840}((210k+79)^2-361)}
+\sum_{k=0}^\infty (-1)^kq^{\frac {1} {840}((210k+89)^2-361)}\\
-\sum_{k=0}^\infty (-1)^kq^{\frac {1} {840}((210k+121)^2-361)}
+\sum_{k=0}^\infty (-1)^kq^{\frac {1} {840}((210k+131)^2-361)}\\
-\sum_{k=0}^\infty (-1)^kq^{\frac {1} {840}((210k+149)^2-361)}
-\sum_{k=0}^\infty (-1)^kq^{\frac {1} {840}((210k+191)^2-361)}\\
=
\sum_{k=0}^\infty (-1)^kq^{\frac{105 k^2}{2}+\frac{19 k}{2}}
+\sum_{k=0}^\infty (-1)^kq^{\frac{105 k^2}{2}+\frac{61 k}{2}+4}
-\sum_{k=0}^\infty (-1)^kq^{\frac{105 k^2}{2}+\frac{79 k}{2}+7}
\kern2cm\\
+\sum_{k=0}^\infty (-1)^kq^{\frac{105 k^2}{2}+\frac{89 k}{2}+9}
-\sum_{k=0}^\infty (-1)^kq^{\frac{105 k^2}{2}+\frac{121 k}{2}+17}
+\sum_{k=0}^\infty (-1)^kq^{\frac{105 k^2}{2}+\frac{131 k}{2}+20}\\
-\sum_{k=0}^\infty (-1)^kq^{\frac{105 k^2}{2}+\frac{149 k}{2}+26}
-\sum_{k=0}^\infty (-1)^kq^{\frac{105 k^2}{2}+\frac{191 k}{2}+43}.
\end{multline}
By performing the replacement $k\to -k-1$, we see that
the last sum on the right-hand side of \eqref{eq:Id5.1A}
can be rewritten as
$$
\sum_{k=0}^\infty (-1)^kq^{\frac{105 k^2}{2}+\frac{191 k}{2}+43}
=
\sum_{k=-\infty}^{-1} (-1)^{k+1}
q^{\frac{105 k^2}{2}+\frac{19 k}{2}}.
$$
Thus, it can be combined with the first sum on the right-hand side of
\eqref{eq:Id5.1A}.
This is similar for the other sums. As a result, they can be paired
so that one obtains four sums over {\it all\/} integers~$k$:
\begin{multline*} 
\sum_{n=0}^\infty(-1)^{t(n)}q^{a_n}=
\sum_{k=-\infty}^\infty (-1)^kq^{\frac{105 k^2}{2}+\frac{19 k}{2}}
+\sum_{k=-\infty}^\infty (-1)^kq^{\frac{105 k^2}{2}+\frac{61 k}{2}+4}\\
-\sum_{k=-\infty}^\infty (-1)^kq^{\frac{105 k^2}{2}+\frac{79 k}{2}+7}
+\sum_{k=-\infty}^\infty (-1)^kq^{\frac{105 k^2}{2}+\frac{89 k}{2}+9}.
\end{multline*}
Now, as announced, 
to each of these sums we apply the Jacobi triple product identity
\eqref{eq:JTP} to get
\begin{multline} \label{eq:840m361}
\sum_{n=0}^\infty(-1)^{t(n)}q^{a_n}=
(q^{105},q^{62},q^{43};q^{105})_\infty
+q^4\,(q^{105},q^{83},q^{22};q^{105})_\infty\\
-q^7\,(q^{105},q^{92},q^{13};q^{105})_\infty
+q^9\,(q^{105},q^{97},q^8;q^{105})_\infty.
\end{multline}
Thus, we have to prove the identity
\begin{multline*} 
0=
(q^{105},q^{62},q^{43};q^{105})_\infty
+q^4\,(q^{105},q^{83},q^{22};q^{105})_\infty\\
-q^7\,(q^{105},q^{92},q^{13};q^{105})_\infty
+q^9\,(q^{105},q^{97},q^8;q^{105})_\infty
-\frac {(q,q^6,q^7;q^7)_\infty} {(q,q^4;q^5)_\infty}.
\end{multline*}
We divide both sides of the identity by the first term on the right-hand side 
and obtain
\begin{multline*} 
0=
1
+q^4\frac{(q^{105},q^{83},q^{22};q^{105})_\infty}
{(q^{105},q^{62},q^{43};q^{105})_\infty}
-q^7\frac{(q^{105},q^{92},q^{13};q^{105})_\infty}
{(q^{105},q^{62},q^{43};q^{105})_\infty}\\
+q^9\frac{(q^{105},q^{97},q^8;q^{105})_\infty}
{(q^{105},q^{62},q^{43};q^{105})_\infty}
-\frac {(q,q^6,q^7;q^7)_\infty} {(q,q^4;q^5)_\infty\,
(q^{105},q^{62},q^{43};q^{105})_\infty}.
\end{multline*}
Now we fix $N:=105$.
With the notation \eqref{eq:E},
our identity can be written as
\begin{equation} \label{eq:0=...} 
0=1+\frac{E_{22}}{E_{43}}-\frac{E_{13}}{E_{43}}
+\frac{E_8}{E_{43}}-\frac{E_7E_8E_{13}E_{15}E_{20}E_{22}
E_{27}E_{28}E_{35}E_{42}E_{48}E_{50}}{E_4E_9E_{11}E_{16}
E_{19}E_{24}E_{26}E_{31}E_{39}E_{44}E_{46}E_{51}}.
\end{equation}
In order to rewrite the last term we used that
$$
(q^7;q^7)=(q^7,q^{14},q^{21},\dots,q^{105};q^{105}),
$$
and similar ``blow-ups" for other terms.

\medskip
{\sc Step 1.} We use the criterion \eqref{eq:Ecrit} to see that all
summands on the right-hand side of \eqref{eq:0=...}
are modular functions for $\Gamma_1(105)$. Indeed, 
for $N=105$ and $\frac{E_{22}}{E_{43}}$, we have 
$1-1\equiv 0\pmod{12}$ and $22^2-43^2\equiv 0\pmod{105}$. 
This shows that $\frac{E_{22}}{E_{43}}$ is a modular function for the
group $\Gamma_1(105)$. 
Similarly, we observe that $\frac{E_{13}}{E_{43}}$, $\frac{E_8}{E_{43}}$ 
and
$\frac{E_7E_8E_{13}E_{15}E_{20}E_{22}E_{27}E_{28}E_{35}E_{42}E_{48}
E_{50}}{E_4E_9E_{11}E_{16}E_{19}E_{24}E_{26}E_{31}E_{39}E_{44}E_{46}E_{51}}$
are modular functions for the group $\Gamma_1(105)$. Consequently, 
$$f:=1+\frac{E_{22}}{E_{43}}-\frac{E_{13}}{E_{43}}+\frac{E_8}{E_{43}}
-\frac{E_7E_8E_{13}E_{15}E_{20}E_{22}E_{27}E_{28}E_{35}E_{42}E_{48}E_{50}}
{E_4E_9E_{11}E_{16}E_{19}E_{24}E_{26}E_{31}E_{39}E_{44}E_{46}E_{51}}$$
is a modular function for the group $\Gamma_1(105)$.

\medskip
{\sc Step 2.}
We use the computer algebra {\sl Magma} to compute
representatives of the cusps for the group $\Gamma_1(105)$. This is
done by using the command \texttt{Cusps(Gamma1(105))}. 
The output is
\begin{verbatim}
[oo,0,1/13,1/12,2/23,1/11,3/32,2/21,1/10,3/29,5/48,2/19,3/28,4/37,1/9,\
5/44,4/35,3/26,8/69,5/43,2/17,3/25,4/33,1/8,6/47,5/39,9/70,4/31,11/84,\
13/99,5/38,7/53,2/15,13/96,8/59,3/22,7/51,1/7,5/34,4/27,29/195,18/121,\
3/20,48/319,79/525,5/33,16/105,7/45,18/115,8/51,4/25,17/105,1/6,6/35,\
11/63,18/103,7/40,8/45,23/129,5/28,7/39,9/50,11/60,9/49,12/65,5/27,\
19/102,30/161,13/69,17/90,4/21,1/5,109/525,27/130,19/91,23/110,22/105,\
47/222,18/85,33/155,3/14,14/65,68/315,13/60,41/189,64/295,29/133,\
19/87,26/119,23/105,9/41,20/91,11/50,11/49,9/40,71/315,23/102,30/133,\
17/75,27/119,8/35,13/56,44/189,7/30,13/55,5/21,6/25,9/35,11/42,59/225,\
37/140,23/87,53/200,13/49,4/15,62/231,51/190,47/175,32/119,368/1365,\
17/63,13/48,29/105,31/112,46/165,41/147,59/210,69/245,2/7,13/45,\
71/245,7/24,92/315,45/154,43/147,31/105,34/115,29/98,52/175,\
25/84,94/315,19/63,32/105,13/42,24/77,11/35,16/45,113/315,48/133,\
38/105,23/63,11/30,31/84,13/35,37/98,8/21,67/175,523/1365,29/75,\
64/165,41/105,124/315,2/5,43/105,41/100,26/63,31/75,44/105,103/245,\
47/105,16/35,7/15,8/15,19/35,39/70,137/245,47/84,17/30,4/7,97/168,\
41/70,37/63,13/21,152/245,87/140,22/35,19/30,24/35,46/63,11/15,\
23/30,27/35]}.
\end{verbatim}

\medskip
{\sc Step 3.} 
We now compute the order of each summand of LHS at each cusp except~$i\infty$.
To do this, we have to use \eqref{eq:ord}.
Assume for example that we want to compute the order at the cusp
$27/35$ of the function $\frac{E_{22}}{E_{43}}$. It is
convenient to implement the following two functions in {\sl Maple}:

\begin{verbatim}
B:=proc(x) 
x^2-x+1/6 
end proc:
\end{verbatim}

\begin{verbatim}
orderCusp1:=proc(x,N,g)
local a,c;
a:=numer(x);
c:=denom(x);
igcd(c,N)*B(frac(a*g/igcd(c,N)))/2
end proc:
\end{verbatim}
Now, in order to achieve our task we enter in {\sl Maple}:
\begin{verbatim}
orderCusp1(27/35,105,22)-orderCusp1(27/35,105,43);
2
\end{verbatim}
The output means that $\frac{E_{22}}{E_{43}}$ has a zero of order $2$
at the cusp $\frac{27}{35}$.  

\medskip
{\sc Step 4.}
We want to compute a lower bound on the order of $f$
at the cusp $27/35$. In order to do this, we do the above computation
for each term in $f$ and take the minimum of all orders. This
computation can be simplified by the function 
\begin{verbatim}
  orderCuspGroup:=proc(x,N,g)
  local ii,order;
  order:=0;
  for ii from 1 to nops(g) do
  order:=order+orderCusp1(x,N,g[ii][1])*g[ii][2];
  od;
  order
  end proc:    
\end{verbatim}
This function takes as input the cusp representative $x$, the $N$ ---
which in our case is $105$ ---, and the index $g$ from $E_g$. The product
$\frac{E_{22}}{E_{43}}$ is expressed as 
\begin{verbatim}
f1:=[[22,1],[43,-1]]; 
\end{verbatim}
Then we can compute the order of $\frac{E_{22}}{E_{43}}$ at the cusp
$27/35$ by typing
\begin{verbatim}
orderCuspGroup(27/35,105,f1);
2
\end{verbatim}
We define the other terms by
\begin{verbatim}
f2:=[[8,1],[43,-1]];
f3:=[[13,1],[43,-1]];
f4:=[[7,1],[8,1],[13,1],[15,1],[20,1],[22,1],[27,1],[28,1],
[35,1],[42,1],[48,1],[50,1],[4,-1],[9,-1],[11,-1],[16,-1],
[19,-1],[24,-1],[26,-1],[31,-1],[39,-1],[44,-1],[46,-1],[51,-1]];    
\end{verbatim}
Now we can get a lower bound on the order of $f$ at the cusp $27/35$
by writing 
\begin{verbatim}
 min(orderCuspGroup(27/35,105,f1),orderCuspGroup(27/35,105,f2),
 orderCuspGroup(27/35,105,f3),orderCuspGroup(27/35,105,f4));
 0
\end{verbatim}
Hence, our lower bound is $0$. 

\medskip
{\sc Step 5.}
We sum up the lower bounds on the orders at all cusps except~$i\infty$.
Thus, we obtain an upper bound on the order of $\tilde f$ at~$[i\infty]$.
This is done as follows:
\begin{verbatim}
cusps:=[0,1/13,1/12,2/23,1/11,3/32,2/21,1/10,3/29,5/48,...];
cnt:=0;
for ii in cusps do 
 mn:=min(orderCuspGroup(ii,105,f1),orderCuspGroup(ii,105,f2),
 orderCuspGroup(ii,105,f3),orderCuspGroup(ii,105,f4)); cnt:=cnt+mn; od; 
print(cnt);
-148
\end{verbatim}
This means that the order of $\tilde f$ at~$i\infty$ is at most $148$,
under the assumption that $f$ is not identically zero. 

\medskip
{\sc Step 6.}
It is
routine to verify that $f=0+0q+\dots+0q^{148}+\cdots$. This implies 
that $\ord_{[i\infty]}(\tilde f)>148$. Hence,
$f$ must be the zero function. 
\end{proof}

\begin{proof}[Proof of Theorem \ref{thm:840m+529}]
It is elementary to see that squares that are congruent to $529$
modulo $840$ are of the form $S^2$, where
$S\equiv23,37,47,103,107,163,173,187$~$(\text{mod}~210)$.
If we now do a computation analogous to the one in the proof of
Theorem~\ref{thm:840m+361}, we obtain
\begin{multline} \label{eq:840m+529} 
\sum_{n=0}^\infty(-1)^{t(n)}q^{a_n}=
(q^{105},q^{64},q^{41};q^{105})_\infty
-q\,(q^{105},q^{71},q^{34};q^{105})_\infty\\
+q^2\,(q^{105},q^{76},q^{29};q^{105})_\infty
+q^{12}\,(q^{105},q^{104},q;q^{105})_\infty.
\end{multline}
Continuing as in the previous proof, we observe 
that the assertion of the theorem is equivalent to
$$0=1-\frac{E_{34}}{E_{41}}+\frac{E_{29}}{E_{41}}+\frac{E_1}{E_{41}}-\frac{E_1E_6E_{15}E_{20}E_{29}E_{34}E_{36}E_{50}E_{14}E_{21}E_{35}E_{49}}{E_2E_3E_{12}E_{17}E_{18}E_{23}E_{32}E_{33}E_{37}E_{38}E_{47}E_{52}},$$
where we used the notation \eqref{eq:E} with $N=105$.
Using \eqref{eq:Ecrit},
we see that the right-hand side is a modular function for the group
$\Gamma_1(105)$. We estimate the sum of the orders of the right-hand
side at the cusps using the same programmes as in the previous proof 
with modifications of the parameters. The
upper bound on the order of the right-hand side at~$i\infty$ is also
$148$, so it suffices to check that the first $148$ coefficients of
the right-hand side are zero in order to prove the identity, which we
checked using a computer.  
\end{proof}

\begin{proof}[Proof of Theorem \ref{thm:840m+121}]
It is elementary to see that squares that are congruent to $121$
modulo $840$ are of the form $S^2$, where
$S\equiv11,31,59,101,109,151,179,199$~$(\text{mod}~210)$.
If we now do a computation analogous to the one in the proof of
Theorem~\ref{thm:840m+361}, we obtain
\begin{multline} \label{eq:840m+121} 
\sum_{n=0}^\infty(-1)^{\fl{(n+4)/8}}q^{a_n}=
(q^{105},q^{47},q^{58};q^{105})_\infty
+q\,(q^{105},q^{37},q^{68};q^{105})_\infty\\
+q^4\,(q^{105},q^{23},q^{82};q^{105})_\infty
+q^{12}\,(q^{105},q^{2},q^{103};q^{105})_\infty.
\end{multline}
Next, we observe that, using the notation \eqref{eq:E} with $N=105$, 
the assertion of the theorem is equivalent to
$$0=1+\frac{E_{37}}{E_{47}}+\frac{E_{23}}{E_{47}}+\frac{E_2}{E_{47}}
-\frac{E_2E_5E_{7}E_{12}E_{23}E_{28}E_{30}E_{33}E_{35}E_{37}E_{40}E_{42}}
{E_1E_4E_{6}E_{11}E_{24}E_{29}E_{31}E_{34}E_{36}E_{39}E_{41}E_{46}}.$$
Using \eqref{eq:Ecrit},
we see that the right-hand side is a modular function for the group
$\Gamma_1(105)$. We estimate the sum of the orders of the right-hand
side at the cusps using the same programmes as before 
with modifications of the parameters. The
upper bound on the order of the right-hand side at~$i\infty$ is also
$148$, so it suffices to check that the first $148$ coefficients of
the right-hand side are zero in order to prove the identity, which we
checked using a computer. 
\end{proof}

\begin{proof}[Proof of Theorem \ref{thm:840m+289}]
It is elementary to see that squares that are congruent to $289$
modulo $840$ are of the form $S^2$, where
$S\equiv17,53,67,73,137,143,157,193$~$(\text{mod}~210)$.
If we now do a computation analogous to the one in the proof of
Theorem~\ref{thm:840m+361}, we obtain
\begin{multline} \label{eq:840m+289} 
\sum_{n=0}^\infty(-1)^{t(n)}q^{a_n}=
(q^{105},q^{44},q^{61};q^{105})_\infty
+q^3\,(q^{105},q^{26},q^{79};q^{105})_\infty\\
-q^5\,(q^{105},q^{19},q^{86};q^{105})_\infty
+q^{6}\,(q^{105},q^{16},q^{89};q^{105})_\infty.
\end{multline}
Next, we observe that, using the notation \eqref{eq:E} with $N=105$, 
the assertion of the theorem is equivalent to
$$0=1+\frac{E_{26}}{E_{44}}-\frac{E_{19}}{E_{44}}+\frac{E_{16}}{E_{44}}-
\frac{E_5E_9E_{14}E_{16}E_{19}E_{21}E_{26}E_{30}E_{35}E_{40}E_{49}E_{51}}
{E_3E_8E_{13}E_{17}E_{18}E_{22}E_{27}E_{32}E_{38}E_{43}E_{48}E_{52}}.$$
Using \eqref{eq:Ecrit},
we see that the right-hand side is a modular function for the group
$\Gamma_1(105)$. We estimate the sum of the orders of the right-hand
side at the cusps using the same programmes as before 
with modifications of the parameters. The
upper bound on the order of the right-hand side at~$i\infty$ is also
$148$, so it suffices to check that the first $148$ coefficients of
the right-hand side are zero in order to prove the identity, which we
checked using a computer. 
\end{proof}

\begin{proof}[Proof of Theorem \ref{thm:840m+1}]
It is elementary to see that squares that are congruent to $1$
modulo $840$ are of the form $S^2$, where
$S\equiv1,29,41,71,139,169,181,209$~$(\text{mod}~210)$.
If we now do a computation analogous to the one in the proof of
Theorem~\ref{thm:840m+361}, we obtain
\begin{multline} \label{eq:840m+1}
\sum_{n=0}^\infty(-1)^{\fl{(n+4)/8}}q^{a_n}=
(q^{105},q^{52},q^{53};q^{105})_\infty
+q\,(q^{105},q^{38},q^{67};q^{105})_\infty\\
+q^2\,(q^{105},q^{32},q^{73};q^{105})_\infty
+q^{6}\,(q^{105},q^{17},q^{88};q^{105})_\infty.
\end{multline}
Next, we observe that, using the notation \eqref{eq:E} with $N=105$, 
the assertion of the theorem is equivalent to
$$0=1+\frac{E_{38}}{E_{52}}+\frac{E_{32}}{E_{52}}+\frac{E_{17}}{E_{52}}
-\frac{E_3E_7E_{10}E_{17}E_{18}E_{25}E_{28}E_{32}E_{35}E_{38}E_{42}E_{45}}
{E_1E_6E_{9}E_{16}E_{19}E_{26}E_{29}E_{34}E_{36}E_{41}E_{44}E_{51}}.$$
Using \eqref{eq:Ecrit},
we see that the right-hand side is a modular function for the group
$\Gamma_1(105)$. We estimate the sum of the orders of the right-hand
side at the cusps using the same programmes as before 
with modifications of the parameters. The
upper bound on the order of the right-hand side at~$i\infty$ is also
$148$, so it suffices to check that the first $148$ coefficients of
the right-hand side are zero in order to prove the identity, which we
checked using a computer.
\end{proof}

\begin{proof}[Proof of Theorem \ref{thm:840m+169}]
It is elementary to see that squares that are congruent to $169$
modulo $840$ are of the form $S^2$, where
$S\equiv13,43,83,97,113,127,167,197$~$(\text{mod}~210)$.
If we now do a computation analogous to the one in the proof of
Theorem~\ref{thm:840m+361}, we obtain
\begin{multline} \label{eq:840m+169}
\sum_{n=0}^\infty(-1)^{t(n)}q^{a_n}=
(q^{105},q^{46},q^{59};q^{105})_\infty
+q^2\,(q^{105},q^{31},q^{74};q^{105})_\infty\\
+q^8\,(q^{105},q^{11},q^{94};q^{105})_\infty
-q^{11}\,(q^{105},q^{4},q^{101};q^{105})_\infty.
\end{multline}
Next, we observe that, using the notation \eqref{eq:E} with $N=105$, 
the assertion of the theorem is equivalent to
$$0=1+\frac{E_{31}}{E_{46}}+\frac{E_{11}}{E_{46}}-\frac{E_{4}}{E_{46}}
-\frac{E_4E_{10}E_{11}E_{14}E_{21}E_{24}E_{25}E_{31}E_{35}E_{39}E_{45}E_{49}}
{E_2E_8E_{12}E_{13}E_{22}E_{23}E_{27}E_{33}E_{37}E_{43}E_{47}E_{48}}.$$
Using \eqref{eq:Ecrit},
we see that the right-hand side is a modular function for the group
$\Gamma_1(105)$. We estimate the sum of the orders of the right-hand
side at the cusps using the same programmes as before 
with modifications of the parameters. The
upper bound on the order of the right-hand side at~$i\infty$ is also
$148$, so it suffices to check that the first $148$ coefficients of
the right-hand side are zero in order to prove the identity, which we
checked using a computer.
\end{proof}

\begin{proof}[Proof of Theorem \ref{thm:240m+1}]
It is elementary to see that squares that are congruent to $1$
modulo $240$ are of the form $S^2$, where
$S\equiv1,31,41,49,71,79,89,119$~$(\text{mod}~120)$.
If we now do a computation analogous to the one in the proof of
Theorem~\ref{thm:840m+361}, we obtain
\begin{multline} \label{eq:240m+1}
\sum_{n=0}^\infty(-1)^{\fl{(n+2)/4}}q^{a_n}=
(q^{120},-q^{59},-q^{61};q^{120})_\infty
+q^4\,(q^{120},-q^{29},-q^{91};q^{120})_\infty\\
-q^7\,(q^{120},-q^{19},-q^{101};q^{120})_\infty
-q^{10}\,(q^{120},-q^{11},-q^{109};q^{120})_\infty.
\end{multline}
Rewriting this expression,
we see that the assertion of the theorem is equivalent to
\begin{multline*}
  0=\frac{(q^{120};q^{120})_{\infty}(q^{118},q^{122};q^{240})_{\infty}}{(q^{59},q^{61};q^{120})_{\infty}}+q^4\frac{(q^{120};q^{120})_{\infty}(q^{58},q^{182};q^{240})_{\infty}}{(q^{29},q^{91};q^{120})_{\infty}}\\-q^7\frac{(q^{120};q^{120})_{\infty}(q^{22},q^{218};q^{240})_{\infty}}{(q^{11},q^{109};q^{120})_{\infty}}-q^{10}\frac{(q^{120};q^{120})_{\infty}(q^{22},q^{218};q^{240})_{\infty}}{(q^{11},q^{109};q^{120})_{\infty}}\\-\frac{(q,q^7,q^8;q^8)_{\infty}(q^6,q^{10};q^{16})_{\infty}}{(q,q^4;q^{5})_{\infty}}.
  \end{multline*}
By dividing both sides of the 
identity by the first term on the right-hand side and using the
notation \eqref{eq:E} with $N=240$, we obtain
\begin{multline*}
  0=1+\frac{E_{58}E_{59}E_{61}}{E_{29}E_{91}E_{118}}-\frac{E_{38}E_{59}E_{61}}{E_{19}E_{101}E_{118}}-\frac{E_{22}E_{59}E_{61}}{E_{11}E_{109}E_{118}}\\
    -\frac{E_7E_8E_{10}E_{15}E_{17}E_{22}E_{23}E_{25}E_{32}E_{33}E_{38}E_{40}E_{42}E_{47}E_{48}E_{55}E_{57}E_{58}}{E_4E_{11}E_{14}E_{19}E_{21}E_{29}E_{34}E_{36}E_{44}E_{46}E_{51}}\\
\times
\frac {E_{63}E_{65}E_{70}E_{72}E_{73}E_{80}E_{87}E_{88}E_{90}E_{95}E_{97}E_{102}E_{103}E_{105}E_{112}E_{113}} {E_{66}E_{69}E_{76}E_{84}E_{91}E_{94}E_{99}E_{101}E_{109}E_{114}E_{116}}.
\end{multline*}
Each term on the right-hand side is a modular function for the group
$\Gamma_1(240)$. As before, using {\sl Magma} we compute a list of all the
cusps. Here we have $448$ cusps. Again, we can give an upper bound
on the order of the right-hand side at~$i\infty$. Running our programme,
we obtain $592$.
We need to verify that the right-hand side has the form
$0+0q+\dots+0q^{592}+\cdots$, which can be routinely done. 
This proves the identity.  
\end{proof}

\begin{proof}[Proof of Theorem \ref{thm:240m+49}]
It is elementary to see that squares that are congruent to $49$
modulo $240$ are of the form $S^2$, where
$S\equiv7,17,23,47,73,97,103,113$~$(\text{mod}~120)$.
If we now do a computation analogous to the one in the proof of
Theorem~\ref{thm:840m+361}, we obtain
\begin{multline} \label{eq:240m+49}
\sum_{n=0}^\infty(-1)^{\fl{5n/4}}q^{a_n}=
(q^{120},-q^{53},-q^{67};q^{120})_\infty
-q\,(q^{120},-q^{43},-q^{77};q^{120})_\infty\\
+q^2\,(q^{120},-q^{37},-q^{83};q^{120})_\infty
-q^{9}\,(q^{120},-q^{13},-q^{107};q^{120})_\infty\\
=\frac{(q^{120};q^{120})_{\infty}(q^{106},q^{134};q^{240})_{\infty}}{(q^{53},q^{67};q^{120})_{\infty}}-q\frac{(q^{120};q^{120})_{\infty}(q^{86},q^{154};q^{240})_{\infty}}{(q^{43},q^{77};q^{120})_{\infty}}\\
  +q^2\frac{(q^{120};q^{120})_{\infty}(q^{74},q^{166};q^{240})_{\infty}}{(q^{37},q^{83};q^{120})_{\infty}}-q^9\frac{(q^{120};q^{120})_{\infty}(q^{26},q^{214};q^{240})_{\infty}}{(q^{13},q^{107};q^{120})_{\infty}}.
\end{multline}
Next, we observe that, using the notation \eqref{eq:E} with $N=240$, 
the assertion of the theorem is equivalent to
\begin{multline*}
  0=1-\frac{E_{53}E_{67}E_{86}}{E_{43}E_{77}E_{106}}+\frac{E_{53}E_{67}E_{74}}{E_{37}E_{83}E_{106}}-\frac{E_{26}E_{53}E_{67}}{E_{13}E_{106}E_{107}}\\
  -\frac{E_1E_6E_9E_{10}E_{15}E_{16}E_{24}E_{25}E_{26}E_{31}E_{39}E_{40}E_{41}E_{49}E_{54}E_{55}E_{56}}{E_2E_3E_{12}E_{13}E_{18}E_{27}E_{28}E_{37}E_{43}E_{52}}\\
\times
\frac
    {E_{64}E_{65}E_{70}E_{71}E_{74}E_{79}E_{80}E_{81}E_{86}E_{89}E_{90}E_{95}E_{96}E_{104}E_{105}E_{111}E_{119}} 
{E_{62}E_{68}E_{77}E_{78}E_{82}E_{83}E_{92}E_{93}E_{98}E_{107}E_{108}E_{117}}.
\end{multline*}
As before, each term is a modular
function for the group $\Gamma_1(240)$. The upper bound 
on the order of the right-hand side at~$i\infty$ is 592. We verified that the
right-hand side has the form $0+0q+\dots+0q^{592}+\cdots$. This proves the
theorem. 
\end{proof}

\begin{proof}[Proof of Theorem \ref{thm:240m+121}]
It is elementary to see that squares that are congruent to $121$
modulo $240$ are of the form $S^2$, where
$S\equiv11,19,29,59,61,91,101,109$~$(\text{mod}~120)$.
If we now do a computation analogous to the one in the proof of
Theorem~\ref{thm:840m+361}, we obtain
\begin{multline} \label{eq:240m+121} 
\sum_{n=0}^\infty(-1)^{\fl{(n+2)/4}}q^{a_n}=
(q^{120},-q^{49},-q^{71};q^{120})_\infty
+q\,(q^{120},-q^{41},-q^{79};q^{120})_\infty\\
-q^3\,(q^{120},-q^{31},-q^{89};q^{120})_\infty
-q^{14}\,(q^{120},-q,-q^{119};q^{120})_\infty\\
=\frac{(q^{120};q^{120})_{\infty}(q^{98},q^{142};q^{240})_{\infty}}{(q^{49},q^{71};q^{120})_{\infty}}+q\frac{(q^{120};q^{120})_{\infty}(q^{82},q^{158};q^{240})_{\infty}}{(q^{41},q^{79};q^{120})_{\infty}}\\
  -q^3\frac{(q^{120};q^{120})_{\infty}(q^{62},q^{178};q^{240})_{\infty}}{(q^{31},q^{89};q^{120})_{\infty}}-q^{14}\frac{(q^{120};q^{120})_{\infty}(q^{2},q^{238};q^{240})_{\infty}}{(q^{1},q^{119};q^{120})_{\infty}}.
\end{multline}
Next, we observe that, using the notation \eqref{eq:E} with $N=240$, 
the assertion of the theorem is equivalent to
\begin{multline*}
  0=1+\frac{E_{49}E_{71}E_{82}}{E_{41}E_{79}E_{98}}-\frac{E_{49}E_{62}E_{71}}{E_{31}E_{89}E_{98}}-\frac{E_{2}E_{49}E_{71}}{E_{1}E_{98}E_{119}}\\
  -\frac{E_2E_3E_5E_{8}E_{13}E_{18}E_{27}E_{30}E_{32}E_{35}E_{37}E_{40}E_{43}E_{45}E_{48}E_{50}E_{53}}{E_1E_4E_{6}E_{9}E_{26}E_{31}E_{36}E_{39}E_{41}E_{44}E_{54}}\\
\times
\frac
    {E_{62}E_{67}E_{72}E_{75}E_{77}E_{78}E_{80}E_{82}E_{83}E_{85}E_{88}E_{93}E_{107}E_{110}E_{112}E_{115}E_{117}}
 {E_{74}E_{76}E_{79}E_{81}E_{84}E_{86}E_{89}E_{106}E_{111}E_{116}E_{119}}.
\end{multline*}
As before, each term is a modular
function for the group $\Gamma_1(240)$. The upper bound 
on the order of the right-hand side at~$i\infty$ is 592. We verified that the
right-hand side has the form $0+0q+\dots+0q^{592}+\cdots$. This proves the
theorem. 
\end{proof}

\begin{proof}[Proof of Theorem \ref{thm:240m+169}]
It is elementary to see that squares that are congruent to $169$
modulo $240$ are of the form $S^2$, where
$S\equiv13,37,43,53,67,77,83,107$~$(\text{mod}~120)$.
If we now do a computation analogous to the one in the proof of
Theorem~\ref{thm:840m+361}, we obtain
\begin{multline} \label{eq:240m+169}
\sum_{n=0}^\infty(-1)^{\fl{5n/4}}q^{a_n}=
(q^{120},-q^{47},-q^{73};q^{120})_\infty
-q^5\,(q^{120},-q^{23},-q^{97};q^{120})_\infty\\
+q^7\,(q^{120},-q^{17},-q^{103};q^{120})_\infty
-q^{11}\,(q^{120},-q^7,-q^{113};q^{120})_\infty\\
=\frac{(q^{120};q^{120})_{\infty}(q^{94},q^{146};q^{240})_{\infty}}{(q^{47},q^{73};q^{120})_{\infty}}-q^5\frac{(q^{120};q^{120})_{\infty}(q^{46},q^{194};q^{240})_{\infty}}{(q^{23},q^{97};q^{120})_{\infty}}\\
  +q^7\frac{(q^{120};q^{120})_{\infty}(q^{34},q^{206};q^{240})_{\infty}}{(q^{17},q^{103};q^{120})_{\infty}}-q^{11}\frac{(q^{120};q^{120})_{\infty}(q^{14},q^{226};q^{240})_{\infty}}{(q^{7},q^{113};q^{120})_{\infty}}.
\end{multline}
Next, we observe that, using the notation \eqref{eq:E} with $N=240$, 
the assertion of the theorem is equivalent to
\begin{multline*}
  0=1-\frac{E_{46}E_{47}E_{73}}{E_{23}E_{94}E_{97}}+\frac{E_{34}E_{47}E_{73}}{E_{17}E_{94}E_{103}}-\frac{E_{14}E_{47}E_{73}}{E_{7}E_{94}E_{113}}\\
  -\frac{E_5E_{11}E_{14}E_{16}E_{19}E_{21}E_{24}E_{29}E_{30}E_{34}E_{35}E_{40}E_{45}E_{46}E_{50}E_{51}E_{56}E_{59}}{E_7E_{12}E_{17}E_{22}E_{23}E_{28}E_{33}E_{38}E_{42}E_{52}E_{57}E_{58}}\\
\times
\frac
    {E_{61}E_{64}E_{66}E_{69}E_{75}E_{80}E_{85}E_{91}E_{96}E_{99}E_{101}E_{104}E_{109}E_{110}E_{114}E_{115}}
 {E_{63}E_{68}E_{87}E_{92}E_{97}E_{102}E_{103}E_{108}E_{113}E_{118}}.
\end{multline*}
As before, each term is a modular
function for the group $\Gamma_1(240)$. The upper bound 
on the order of the right-hand side at~$i\infty$ is 592. We verified that the
right-hand side has the form $0+0q+\dots+0q^{592}+\cdots$. This proves the
theorem. 
\end{proof}

\begin{proof}[Proof of Theorem \ref{thm:15m+1}]
This is a special case of Theorem~\ref{thm:3Pm+a^2}.
\end{proof}

\begin{proof}[Proof of Theorem \ref{thm:15m+4}]
This is a special case of Theorem~\ref{thm:3Pm+a^2}.
\end{proof}

\begin{proof}[Proof of Theorem \ref{thm:6.7}]
This is a direct consequence of the Jacobi triple product
identity \eqref{eq:JTP}: one replaces $q$ by $q^{10}$ and then
chooses $z=-q^{6}$ there.
\end{proof}

\begin{proof}[Proof of Theorem \ref{thm:6.8}]
Using \eqref{eq:n(n+1)}, we obtain
\begin{equation} \label{eq:6.8} 
\sum_{n=0}^\infty \big(q^{n(n+1)}- q^{5n(n+1)+1}\big)
=
(q^2;q^2)_\infty\,(-q^2;q^2)_\infty^2
-q\,(q^{10};q^{10})_\infty\,(-q^{10};q^{10})_\infty^2.
\end{equation}
Hence, the assertion of the theorem is equivalent to
$$0=\frac{(q^4;q^4)_{\infty}^2}{(q^2;q^2)_{\infty}}-q\frac{(q^{20};q^{20})_{\infty}^2}{(q^{10};q^{10})_{\infty}}-\frac {(q,q^9,q^{10};q^{10})_\infty\,(q^8,q^{12};q^{20})_\infty}
{(q^2,q^3;q^5)_\infty}.$$
Using the notation \eqref{eq:E} with $N=20$, 
we see that this is equivalent to
$$0=1-\frac{E_2E_6}{E_4E_8}-\frac{E_1E_6E_9E_{10}}{E_3E_4E_7E_8}.$$
The right-hand side is a modular function for the group
$\Gamma_1(20)$. 
The cusps for this group are computed as usual using {\sl Magma}:
$$\left\{\infty,0,\frac{1}{7},\frac{3}{20},\frac{1}{6},\frac{2}{11},\frac{1}{5},\frac{1}{4},\frac{3}{10},\frac{1}{3},\frac{7}{20},\frac{11}{30},\frac{3}{8},\frac{2}{5},\frac{9}{20},\frac{1}{2},\frac{7}{12},\frac{3}{5},\frac{3}{4},\frac{4}{5}\right\}.$$
There are in total $20$ cusps. Estimating the order of the right-hand
side at~$i\infty$, one obtains an upper bound of $4$. Hence
it is sufficient to show that the right-hand side has the form
$0+0q+0q^2+0q^3+0q^4+\cdots$, which can be done routinely.  
\end{proof}

\begin{proof}[Proof of Theorem \ref{thm:6.9}]
Using \eqref{eq:n^2}, we obtain
\begin{equation} \label{eq:6.9} 
1+\sum_{n=1}^\infty\big( q^{n^2}+q^{5n^2}\big)
=\frac {1} {2}\big((q^2;q^2)_\infty\,(-q;q^2)_\infty^2
+(q^{10};q^{10})_\infty\,(-q^5;q^{10})_\infty^2\big).
\end{equation}
Hence, the assertion of the theorem is equivalent to
$$0=\frac{1}{2}\frac{(q^2;q^2)_{\infty}(q^2;q^4)_{\infty}^2}{(q;q^2)^2_{\infty}}+\frac{1}{2}\frac{(q^{10};q^{10})_{\infty}(q^{10};q^{20})^2_{\infty}}{(q^5;q^{10})^2_{\infty}}-\frac {(q^2,q^8,q^{10};q^{10})_\infty\,(q^6,q^{14};q^{20})_\infty}
{(q,q^4;q^5)_\infty}.$$
Using the notation \eqref{eq:E} with $N=20$, 
we see that this is equivalent to
$$0=\frac{1}{2}+\frac{1}{2}\frac{E_1^2E_3^2E_7^2E_9^2}{E_2^3E_4E_6^3E_8}-\frac{E_1E_3^2E_5^2E_7^2E_9}{E_2^2E_4^2E_6^3E_{10}}.$$
The right-hand side is a modular function for the group $\Gamma_1(20)$.
There are in total the $20$ cusps exhibited in the previous proof. 
Estimating the order of the right-hand
side at~$i\infty$, one obtains an upper bound of $4$. Hence
it is sufficient to show that the right-hand side has the form
$0+0q+0q^2+0q^3+0q^4+\cdots$, which can be done routinely.  
\end{proof}

\begin{proof}[Proof of Theorem \ref{thm:6.10}]
This is a direct consequence of the Jacobi triple product
identity \eqref{eq:JTP}: one replaces $q$ by $q^{10}$ and then
chooses $z=-q^{7}$ there.
\end{proof}

\begin{proof}[Proof of Theorem \ref{thm:6.11}]
Using \eqref{eq:n(n+1)}, we obtain
\begin{equation} \label{eq:6.11} 
  \sum_{n=0}^\infty \big(q^{n(n+1)}+ q^{5n(n+1)+1}\big)
=
(q^2;q^2)_\infty\,(-q^2;q^2)_\infty^2
+q\,(q^{10};q^{10})_\infty\,(-q^{10};q^{10})_\infty^2.
\end{equation}
Hence, the assertion of the theorem is equivalent to
$$0=\frac{(q^4;q^4)_{\infty}^2}{(q^2;q^2)_{\infty}}+q\frac{(q^{20};q^{20})_{\infty}^2}{(q^{10};q^{10})_{\infty}}-\frac{(q^3,q^7,q^{10};q^{10})_{\infty}(q^4,q^{16};q^{20})_{\infty}}{(q,q^4;q^5)_{\infty}}.$$
Using the notation \eqref{eq:E} with $N=20$, 
we see that this is equivalent to
$$0=1+\frac{E_2E_6}{E_4E_8}-\frac{E_2E_3E_7E_{10}}{E_1E_4E_8E_9}.$$
The right-hand side is a modular function for the group $\Gamma_1(20)$.
There are in total the $20$ cusps exhibited in the proof of Theorem~\ref{thm:6.9}. 
Estimating the order of the right-hand
side at~$i\infty$, one obtains an upper bound of $4$. Hence
it is sufficient to show that the right-hand side has the form
$0+0q+0q^2+0q^3+0q^4+\cdots$, which can be done routinely.  
\end{proof}

\begin{proof}[Proof of Theorem \ref{thm:6.12}]
This is a direct consequence of the Jacobi triple product
identity \eqref{eq:JTP}: one replaces $q$ by $q^{10}$ and then
chooses $z=-q^{8}$ there.
\end{proof}

\begin{proof}[Proof of Theorem \ref{thm:6.13}]
This is a direct consequence of the Jacobi triple product
identity \eqref{eq:JTP}: one replaces $q$ by $q^{10}$ and then
chooses $z=-q^{9}$ there.
\end{proof}

\begin{proof}[Proof of Theorem \ref{thm:6.14}]
Using \eqref{eq:n^2}, we obtain
\begin{equation} \label{eq:6.14} 
  \sum_{n=1}^\infty \big(q^{n^2-1}- q^{5n^2-1}\big)
=\frac {1} {2q}\big((q^2;q^2)_\infty\,(-q;q^2)_\infty^2
-(q^{10};q^{10})_\infty\,(-q^5;q^{10})_\infty^2\big).
\end{equation}
Hence, the assertion of the theorem is equivalent to
$$0=\frac{1}{2q}\frac{(q^2;q^2)_{\infty}(q^2;q^4)_{\infty}^2}{(q;q^2)^2_{\infty}}-\frac{1}{2q}\frac{(q^{10};q^{10})_{\infty}(q^{10};q^{20})_{\infty}^2}{(q^5;q^{10})_{\infty}^2}-\frac{(q^4,q^6,q^{10};q^{10})_{\infty}(q^2,q^{18};q^{20})_{\infty}}{(q^2;q^3;q^5)_{\infty}}.$$
Using the notation \eqref{eq:E} with $N=20$, 
we see that this is equivalent to
$$0=\frac{1}{2}-\frac{1}{2}\frac{E_1^2E_3^2E_7^2E_9^2}{E_2^3E_4E_6^3E_8}-\frac{E_1^2E_3E_5^2E_7E_9^2}{E_2^3E_6^2E_8^2E_{10}}.$$
The right-hand side is a modular function for the group $\Gamma_1(20)$.
There are in total the $20$ cusps exhibited in the proof of Theorem~\ref{thm:6.9}. 
Estimating the order of the right-hand
side at~$i\infty$, one obtains an upper bound of $4$. Hence
it is sufficient to show that the right-hand side has the form
$0+0q+0q^2+0q^3+0q^4+\cdots$, which can be done routinely.  
\end{proof}

\begin{proof}[Proof of Theorem \ref{thm:120m+49}]
This is a special case of Corollary~\ref{cor:24Pm+a^2}.
\end{proof}

\begin{proof}[Proof of Theorem \ref{thm:120m+1}]
This is a special case of Theorem~\ref{thm:24Pm+a^2}.
\end{proof}

\begin{proof}[Proof of Theorem \ref{thm:168m+121}]
This is a special case of Corollary~\ref{cor:24Pm+a^2}.
\end{proof}

\begin{proof}[Proof of Theorem \ref{thm:168m+1}]
This is a special case of Theorem~\ref{thm:24Pm+a^2}.
\end{proof}

\begin{proof}[Proof of Theorem \ref{thm:168m+25}]
This is a special case of Theorem~\ref{thm:24Pm+a^2}.
\end{proof}

\begin{proof}[Proof of Theorem \ref{thm:48m+1}]
It is elementary to see that squares that are congruent to $1$
modulo $48$ are of the form $N^2$, where
$N\equiv1,7,17,23$~$(\text{mod}~24)$.
If we now do a computation analogous to the one in the proof of
Theorem~\ref{thm:840m+361}, we obtain
\begin{equation} \label{eq:48m+1} 
\sum_{n=0}^\infty(-1)^{\fl{(n+2)/4}}q^{a_n}
=
(q^{24},q^{13},q^{11};q^{24})_\infty
+q(q^{24},q^{19},q^{5};q^{24})_\infty.
\end{equation}
Next, we observe that, using the notation \eqref{eq:E} with $N=24$, 
the assertion of the theorem is equivalent to
$$0=1+\frac{E_5}{E_{11}}-\frac{E_2E_6E_8E_{10}}{E_1E_7E_9E_{11}}.$$
The right-hand side is a modular function for the group
$\Gamma_1(24)$. 
The cusps for this group are computed as usual using {\sl Magma}:
$$\left\{\infty,0,\frac{1}{8},\frac{1}{7},\frac{1}{6},\frac{2}{11},\frac{1}{5},\frac{5}{24},\frac{3}{14},\frac{2}{9},\frac{1}{4},\frac{7}{24},\frac{11}{36},\frac{1}{3},\frac{3}{8},\frac{7}{18},\frac{5}{12},\frac{4}{9},\frac{11}{24},\frac{1}{2},\frac{9}{16},\frac{5}{8},\frac{2}{3},\frac{3}{4}\right\}.$$
There are in total $24$ cusps. 
Estimating the order of the right-hand
side at~$i\infty$, one obtains an upper bound of $4$. Hence
it is sufficient to show that the right-hand side has the form
$0+0q+0q^2+0q^3+0q^4+\cdots$, which can be done routinely.  
\end{proof}

\begin{proof}[Proof of Theorem \ref{thm:48m+25}]
It is elementary to see that squares that are congruent to $25$
modulo $48$ are of the form $S^2$, where
$S\equiv5,11,13,19$~$(\text{mod}~24)$.
If we now do a computation analogous to the one in the proof of
Theorem~\ref{thm:840m+361}, we obtain
\begin{equation} \label{eq:48m+25} 
\sum_{n=0}^\infty(-1)^{\fl{5n/4}}q^{a_n}
=
(q^{24},q^{17},q^{7};q^{24})_\infty
-q^{2}(q^{24},q^{23},q;q^{24})_\infty.
\end{equation}
Next, we observe that, using the notation \eqref{eq:E} with $N=24$, 
the assertion of the theorem is equivalent to
$$0=1-\frac{E_1}{E_{7}}-\frac{E_2E_6E_8E_{10}}{E_3E_5E_7E_{11}}.$$
The right-hand side is a modular function for the group
$\Gamma_1(24)$. 
There are in total the $24$ cusps exhibited in the previous proof. 
Estimating the order of the right-hand
side at~$i\infty$, one obtains an upper bound of $4$. Hence
it is sufficient to show that the right-hand side has the form
$0+0q+0q^2+0q^3+0q^4+\cdots$, which can be done routinely.  
\end{proof}

\begin{proof}[Proof of Theorem \ref{thm:21m+1}]
This is a special case of Theorem~\ref{thm:3Pm+a^2}.
\end{proof}

\begin{proof}[Proof of Theorem \ref{thm:21m+4}]
This is a special case of Theorem~\ref{thm:3Pm+a^2}.
\end{proof}

\begin{proof}[Proof of Theorem \ref{thm:21m+16}]
This theorem is a special case of Corollary~\ref{cor:3Pm+a^2}.
\end{proof}

\begin{proof}[Proof of Theorem \ref{thm:16m+1}]
This is a special case of Theorem~\ref{thm:16m+a^2}.
\end{proof}

\begin{proof}[Proof of Theorem \ref{thm:16m+9}]
This is a special case of Theorem~\ref{thm:16m+a^2}.
\end{proof}

\section{Theta function identities}
\label{sec:theta}

The purpose of this section is, first of all, to present
Weierstra{\ss}' addition formula for theta functions, and, second,
to make two special cases explicit that are
particularly used in the proofs of our theorems from
Sections~\ref{sec:Ids1} and~\ref{sec:Ids2} given in the 
next section, and also in the proofs in
Section~\ref{sec:Ids3}.

\medskip
In this section and the following ones, we use a different notation
for the theta functions that appear in our context, namely
$$
\theta(\alpha;q):=(\alpha,q/\alpha;q)_\infty.
$$
It should be noted that, up to a power of~$q$,
the function $E_g(q;N)$ that we used in Sections~\ref{sec:auto}
and \ref{sec:mech} can be expressed as $\theta(q^g;q^N)$.

\medskip
Using the above notation, Weierstra\ss' addition formula (cf.\ 
\cite[p.~451, Example~5]{WW}) reads
\begin{multline}\label{eq:tadd}
\theta(xy;q)\,\theta(x/y;q)\,\theta(uv;q)\,\theta(u/v;q)-
\theta(xv;q)\,\theta(x/v;q)\,\theta(uy;q)\,\theta(u/y;q)\\
=\frac uy\,\theta(yv;q)\,\theta(y/v;q)\,\theta(xu;q)\,\theta(x/u;q).
\end{multline}

\medskip
Two specialisations of this formula are of particular importance in our
context.
If, in \eqref{eq:tadd}, we replace $q$ by $q^{3N}$
and specialise $x=q^N$, $y=u^2/q^N$, and
$v=q^N/u$, then we obtain the relation
\begin{multline*}
\theta(u^2;q^{3N})\,\theta(q^{2N}/u^2;q^{3N})\,\theta(q^N;q^{3N})\,\theta(u^2/q^N;q^{3N})\\
-
\theta(q^{2N}/u;q^{3N})\,\theta(u;q^{3N})\,\theta(u^3/q^N;q^{3N})\,\theta(q^N/u;q^{3N})\\
=\frac {q^N} {u}\theta(u;q^{3N})\,\theta(u^3/q^{2N};q^{3N})\,\theta(q^Nu;q^{3N})\,\theta(q^N/u;q^{3N}),
\end{multline*}
or, equivalently,
\begin{multline*} 
\theta(u^3/q^N;q^{3N})
+\frac {q^N} {u}\,\theta(u^3/q^{2N};q^{3N})\\
=
\frac
{\theta(u^2;q^{3N})\,\theta(q^{2N}/u^2;q^{3N})\,\theta(q^N;q^{3N})\,\theta(u^2/q^N;q^{3N})}
{\theta(u;q^{3N})\,\theta(q^N/u;q^{3N})\,\theta(q^Nu;q^{3N})}.
\end{multline*}
Written in alternative notation, this is
\begin{multline} \label{eq:threl1}
(u^3/q^N,q^{4N}/u^3,q^{3N};q^{3N})_\infty
+\frac {q^N} {u}\,(u^3/q^{2N},q^{5N}/u^3,q^{3N};q^{3N})_\infty\\
=
\frac
{(u^2/q^N,q^{2N}/u^2,q^N;q^{N})_\infty}
{(u,q^N/u;q^{N})_\infty}.
\end{multline}

\medskip
Similarly, if in \eqref{eq:tadd} we replace $q$ by $q^{3N}$
and specialise $x=q^{2N}$, $y=u^2/q^{2N}$, and
$v=q^{2N}/u$, then we obtain the relation
\begin{multline*}
\theta(u^2;q^{3N})\,\theta(q^{4N}/u^2;q^{3N})\,\theta(q^{2N};q^{3N})\,\theta(u^2/q^{2N};q^{3N})\\
-
\theta(q^{4N}/u;q^{3N})\,\theta(u;q^{3N})\,\theta(u^3/q^{2N};q^{3N})\,\theta(q^{2N}/u;q^{3N})\\
=\frac {q^{2N}} {u}\theta(u;q^{3N})\,\theta(u^3/q^{4N};q^{3N})\,\theta(q^{2N}u;q^{3N})\,\theta(q^{2N}/u;q^{3N}),
\end{multline*}
or, equivalently,
\begin{multline*} 
\theta(u^3/q^{2N};q^{3N})-
\frac {q^{3N}} {u^2}\theta(u^3/q^{4N};q^{3N})\\
=
\frac
{\theta(u^2;q^{3N})\,\theta(q^{4N}/u^2;q^{3N})\,\theta(q^{2N};q^{3N})\,\theta(u^2/q^{2N};q^{3N})}
{\theta(u;q^{3N})\,\theta(q^{2N}/u;q^{3N})\,\theta(u/q^N;q^{3N})}.
\end{multline*}
Written in alternative notation, this is
\begin{multline} \label{eq:threl2}
(u^3/q^{2N},q^{5N}/u^3,q^{3N};q^{3N})_\infty-
\frac {q^{3N}} {u^2}\,(u^3/q^{4N},q^{7N}/u^3,q^{3N};q^{3N})_\infty\\
=
\frac
{(u^2/q^{2N},q^{3N}/u^2,q^N;q^{N})_\infty}
{(u/q^N,q^{2N}/u;q^{N})_\infty}.
\end{multline}

\section{Proofs by using the Weierstra{\ss} relation}
\label{sec:Riemann}

In this section, we provide proofs of the
theorems in Sections~\ref{sec:Ids1} and~\ref{sec:Ids2} that 
utilise the Weierstra{\ss} relation \eqref{eq:tadd}.
Again, for the theorems which are
specialisations of the parametric theorems in Section~\ref{sec:Ids3},
we refer to the proofs given there (which also make use the
Weierstra{\ss} relation).

\begin{proof}[Proof of Theorem \ref{thm:840m+361}]
Our point of departure is \eqref{eq:840m361}. 
By \eqref{eq:threl1} with $N=35$ and $u=q^{26}$, respectively with
$N=35$ and $u=q^{19}$, we get
%
%
$$
\sum_{n=0}^\infty(-1)^{t(n)}q^{a_n}=
\frac
{(q^{17},q^{18},q^{35};q^{35})_\infty}
{(q^{26},q^{9};q^{35})_\infty}
+
q^4
\frac
{(q^{3},q^{32},q^{35};q^{35})_\infty}
{(q^{19},q^{16};q^{35})_\infty}.
$$
If we now replace $q$ by $q^{35}$ and 
choose $u=q^{10}$, $v=q^3$, $x=q^{14}$, and $y=q^6$
in \eqref{eq:tadd}, we obtain\footnote{In case the reader wonders how
we came up with these choices of $u,v,x,y$ (and the choices in
subsequent proofs): after a lot of trial and error which produced
some useful choices in certain cases, 
but did not lead to the recognition of 
any underlying patterns (we doubt in fact
that there are),
we decided to write a {\sl Maple} programme that goes through all possible
choices of $u,v,x,y$ in non-negative powers of~$q$ and outputs 
choices that are appropriate to establish our identities.}
\begin{multline*}
\theta(q^{17};q^{35})\, \theta(q^{11};q^{35})\,
\theta(q^{16};q^{35})\, \theta(q^4;q^{35})
+ q^4 \theta(q^{9};q^{35})\, \theta(q^{3};q^{35})\,
 \theta(q^{24};q^{35})\, \theta(q^4;q^{35}) \\
=
\theta(q^{20};q^{35})\, \theta(q^{8};q^{35})\, \theta(q^{13};q^{35})\,
\theta(q^7;q^{35}),
\end{multline*}
and thus the above right-hand side becomes
$$
\frac {\theta(q^{20};q^{35})\, \theta(q^{8};q^{35})\, \theta(q^{13};q^{35})\,
\theta(q^7;q^{35})\,(q^{35};q^{35})_\infty} 
{\theta(q^{16};q^{35})\,\theta(q^{9};q^{35})\,\theta(q^{11};q^{35})\,
\theta(q^{4};q^{35})},
$$
which is equivalent to the right-hand side of \eqref{eq:Id5.1}.
\end{proof}

\begin{proof}[Proof of Theorem \ref{thm:840m+529}]
Our point of departure is \eqref{eq:840m+529}. 
By \eqref{eq:threl1} with $N=35$ and $u=q^{33}$, respectively with
$N=35$ and $u=q^{23}$, we get
$$
\sum_{n=0}^\infty(-1)^{t(n)}q^{a_n}=
\frac
{(q^{31},q^{4},q^{35};q^{35})_\infty}
{(q^{33},q^{2};q^{35})_\infty}
-
q
\frac
{(q^{11},q^{24},q^{35};q^{35})_\infty}
{(q^{23},q^{12};q^{35})_\infty}.
$$
If we now replace $q$ by $q^{35}$ and choose
$u = q^{15}$, $v = q^3$, $x = q^{17}$, and $y = q^{14}$
in \eqref{eq:tadd}, we obtain
\begin{multline*}
\theta(q^{31};q^{35})\,
\theta(q^3;q^{35})\,
\theta(q^{18};q^{35})\,
\theta(q^{12};q^{35})
-
\theta(q^{20};q^{35})\,
\theta(q^{14};q^{35})\,
\theta(q^{29};q^{35})\,
\theta(q;q^{35})
\\
=
q
\theta(q^{17};q^{35})\,
\theta(q^{11};q^{35})\,
\theta(q^{32};q^{35})\,
\theta(q^2;q^{35}),
\end{multline*}
and thus the above right-hand side becomes
$$
\frac {\theta(q^{20};q^{35})\,
\theta(q^{14};q^{35})\,
\theta(q^{29};q^{35})\,
\theta(q;q^{35})\,
(q^{35};q^{35})_\infty}
{\theta(q^{17};q^{35})\,
\theta(q^{12};q^{35})\,
\theta(q^{32};q^{35})\,
\theta(q^2;q^{35})},
$$
which is equivalent to the right-hand side of \eqref{eq:Id5.2}.
\end{proof}

\begin{proof}[Proof of Theorem \ref{thm:840m+121}]
Our point of departure is \eqref{eq:840m+121}. 
By \eqref{eq:threl1} with $N=35$ and $u=q^{31}$, respectively with
$N=35$ and $u=q^{24}$, we get
$$
\sum_{n=0}^\infty(-1)^{\fl{(n+4)/8}}q^{a_n}=
\frac
{(q^{27},q^{8},q^{35};q^{35})_\infty}
{(q^{31},q^{4};q^{35})_\infty}
+
q
\frac
{(q^{13},q^{22},q^{35};q^{35})_\infty}
{(q^{24},q^{11};q^{35})_\infty}.
$$
If we now replace $q$ by $q^{35}$ and choose
$u = q^{14}$, $v = q^9$, $x = q^{15}$, and $y = q^{13}$
in \eqref{eq:tadd}, we obtain
\begin{multline*}
\theta(q^{28};q^{35})\,
\theta(q^2;q^{35})\,
\theta(q^{23};q^{35})\,
\theta(q^{5};q^{35})
-
\theta(q^{24};q^{35})\,
\theta(q^{6};q^{35})\,
\theta(q^{27};q^{35})\,
\theta(q;q^{35})
\\
=
q
\theta(q^{2};q^{35})\,
\theta(q^{4};q^{35})\,
\theta(q^{29};q^{35})\,
\theta(q;q^{35}),
\end{multline*}
and thus the above right-hand side becomes
$$
\frac {\theta(q^{28};q^{35})\,
\theta(q^2;q^{35})\,
\theta(q^{23};q^{35})\,
\theta(q^{5};q^{35})
(q^{35};q^{35})_\infty}
{\theta(q^{4};q^{35})\,
\theta(q^{11};q^{35})\,
\theta(q^{6};q^{35})\,
\theta(q;q^{35})},
$$
which is equivalent to the right-hand side of \eqref{eq:Id5.3}.
\end{proof}

\begin{proof}[Proof of Theorem \ref{thm:840m+289}]
Our point of departure is \eqref{eq:840m+289}. 
By \eqref{eq:threl1} with $N=35$ and $u=q^{32}$, respectively with
$N=35$ and $u=q^{18}$, we get
$$
\sum_{n=0}^\infty(-1)^{t(n)}q^{a_n}=
\frac
{(q^{29},q^{6},q^{35};q^{35})_\infty}
{(q^{32},q^{3};q^{35})_\infty}
-
q^5
\frac
{(q,q^{34},q^{35};q^{35})_\infty}
{(q^{18},q^{17};q^{35})_\infty}.
$$
If we now replace $q$ by $q^{35}$ and choose
$u = q^{7}$, $v = q^{15}$, $x = q$, and $y = q^{2}$
in \eqref{eq:tadd}, we obtain
\begin{multline*}
\theta(q^{17};q^{35})\,
\theta(q^8;q^{35})\,
\theta(q^{6};q^{35})\,
\theta(q^{13};q^{35})
-
\theta(q^{9};q^{35})\,
\theta(q^{16};q^{35})\,
\theta(q^{14};q^{35})\,
\theta(q^5;q^{35})
\\
=
q^5
\theta(q;q^{35})\,
\theta(q^{8};q^{35})\,
\theta(q^{22};q^{35})\,
\theta(q^3;q^{35})
\end{multline*}
after little manipulation.
Thus, the above right-hand side becomes
$$
\frac {\theta(q^{9};q^{35})\,
\theta(q^{16};q^{35})\,
\theta(q^{14};q^{35})\,
\theta(q^5;q^{35})
(q^{35};q^{35})_\infty}
{\theta(q^{17};q^{35})\,
\theta(q^{3};q^{35})\,
\theta(q^{13};q^{35})\,
\theta(q^8;q^{35})},
$$
which is equivalent to the right-hand side of \eqref{eq:Id5.4}.
\end{proof}

\begin{proof}[Proof of Theorem \ref{thm:840m+1}]
Our point of departure is \eqref{eq:840m+1}. 
By \eqref{eq:threl1} with $N=35$ and $u=q^{29}$, respectively with
$N=35$ and $u=q^{34}$, we get
$$
\sum_{n=0}^\infty(-1)^{\fl{(n+4)/8}}q^{a_n}=
\frac
{(q^{23},q^{12},q^{35};q^{35})_\infty}
{(q^{29},q^{6};q^{35})_\infty}
+
q
\frac
{(q^{33},q^{2},q^{35};q^{35})_\infty}
{(q^{34},q;q^{35})_\infty}.
$$
If we now replace $q$ by $q^{35}$ and choose
$u = q^{5}$, $v = q^{14}$, $x = q^{2}$, and $y = q^{4}$
in \eqref{eq:tadd}, we obtain
\begin{multline*}
\theta(q^{18};q^{35})\,
\theta(q^3;q^{35})\,
\theta(q^{10};q^{35})\,
\theta(q^{7};q^{35})
-
\theta(q^{12};q^{35})\,
\theta(q^{9};q^{35})\,
\theta(q^{16};q^{35})\,
\theta(q;q^{35})
\\
=
q
\theta(q^{9};q^{35})\,
\theta(q^{6};q^{35})\,
\theta(q^{19};q^{35})\,
\theta(q^2;q^{35})
\end{multline*}
after little manipulation.
Thus, the above right-hand side becomes
$$
\frac {\theta(q^{18};q^{35})\,
\theta(q^3;q^{35})\,
\theta(q^{10};q^{35})\,
\theta(q^{7};q^{35})
(q^{35};q^{35})_\infty}
{\theta(q^{6};q^{35})\,
\theta(q;q^{35})\,
\theta(q^{16};q^{35})\,
\theta(q^9;q^{35})},
$$
which is equivalent to the right-hand side of \eqref{eq:Id5.5}.
\end{proof}

\begin{proof}[Proof of Theorem \ref{thm:840m+169}]
Our point of departure is \eqref{eq:840m+169}. 
By \eqref{eq:threl1} with $N=35$ and $u=q^{27}$, respectively with
$N=35$ and $u=q^{22}$, we get
$$
\sum_{n=0}^\infty(-1)^{t(n)}q^{a_n}=
\frac
{(q^{19},q^{16},q^{35};q^{35})_\infty}
{(q^{27},q^{8};q^{35})_\infty}
+
q^2
\frac
{(q^{9},q^{26},q^{35};q^{35})_\infty}
{(q^{22},q^{13};q^{35})_\infty}.
$$
If we now replace $q$ by $q^{35}$ and choose
$u = q^{7}$, $v = q^{16}$, $x = q^{3}$, and $y = q^{5}$
in \eqref{eq:tadd}, we obtain
\begin{multline*}
\theta(q^{21};q^{35})\,
\theta(q^4;q^{35})\,
\theta(q^{11};q^{35})\,
\theta(q^{10};q^{35})
-
\theta(q^{13};q^{35})\,
\theta(q^{12};q^{35})\,
\theta(q^{19};q^{35})\,
\theta(q^2;q^{35})
\\
=
q^2
\theta(q^{9};q^{35})\,
\theta(q^{8};q^{35})\,
\theta(q^{23};q^{35})\,
\theta(q^2;q^{35})
\end{multline*}
after little manipulation.
Thus, the above right-hand side becomes
$$
\frac {\theta(q^{21};q^{35})\,
\theta(q^4;q^{35})\,
\theta(q^{11};q^{35})\,
\theta(q^{10};q^{35})
(q^{35};q^{35})_\infty}
{\theta(q^{8};q^{35})\,
\theta(q^{13};q^{35})\,
\theta(q^{12};q^{35})\,
\theta(q^2;q^{35})},
$$
which is equivalent to the right-hand side of \eqref{eq:Id5.6}.
\end{proof}

\begin{proof}[Proof of Theorem \ref{thm:240m+1}]
Our point of departure is \eqref{eq:240m+1}. 
By \eqref{eq:threl1} with $N=40$ and $u=-q^{33}$, respectively with
$N=40$ and $u=-q^{23}$, we get
$$
\sum_{n=0}^\infty(-1)^{\fl{(n+2)/4}}q^{a_n}=
\frac
{(q^{26},q^{14},q^{40};q^{40})_\infty}
{(-q^{33},-q^{7};q^{40})_\infty}
+
q^4
\frac
{(q^{6},q^{34},q^{40};q^{40})_\infty}
{(-q^{23},-q^{17};q^{40})_\infty}.
$$
If we now replace $q$ by $q^{40}$ and choose
$u = -q^{9}$, $v = q^{16}$, 
$x = -q$, and $y = -q^{5}$
in \eqref{eq:tadd}, we obtain
\begin{multline*}
\theta(-q^{21};q^{40})\,
\theta(q^8;q^{40})\,
\theta(q^{10};q^{40})\,
\theta(-q^{11};q^{40})\\
-
\theta(q^{14};q^{40})\,
\theta(-q^{15};q^{40})\,
\theta(-q^{17};q^{40})\,
\theta(q^4;q^{40})
\\
=
q^4
\theta(q^{6};q^{40})\,
\theta(-q^{7};q^{40})\,
\theta(-q^{25};q^{40})\,
\theta(q^4;q^{40})
\end{multline*}
after little manipulation.
Thus, the above right-hand side becomes
$$
\frac {\theta(-q^{21};q^{40})\,
\theta(q^8;q^{40})\,
\theta(q^{10};q^{40})\,
\theta(-q^{11};q^{40})
\theta(q;q^{40})\,
(q^{40};q^{40})_\infty}
{\theta(-q^{17};q^{40})\,
\theta(-q^{7};q^{40})\,
\theta(-q^{15};q^{40})\,
\theta(q^4;q^{40})},
$$
which is equivalent to the right-hand side of \eqref{eq:Id6.1}.
\end{proof}

\begin{proof}[Proof of Theorem \ref{thm:240m+49}]
Our point of departure is the first two lines in \eqref{eq:240m+49}. 
By \eqref{eq:threl1} with $N=40$ and $u=-q^{31}$, respectively with
$N=40$ and $u=-q^{39}$, we get
$$
\sum_{n=0}^\infty(-1)^{\fl{5n/4}}q^{a_n}=
\frac
{(q^{22},q^{18},q^{40};q^{40})_\infty}
{(-q^{31},-q^{9};q^{40})_\infty}
-
q
\frac
{(q^{38},q^{2},q^{40};q^{40})_\infty}
{(-q^{39},-q;q^{40})_\infty}.
$$
If we now replace $q$ by $q^{40}$ and choose
$u = -q^{8}$, $v = q^5$, $x = q^{17}$, and $y = q^{7}$
in \eqref{eq:tadd}, we obtain
\begin{multline*}
\theta(q^{24};q^{40})\,
\theta(q^{10};q^{40})\,
\theta(-q^{13};q^{40})\,
\theta(-q^{3};q^{40})\\
-
\theta(q^{22};q^{40})\,
\theta(q^{12};q^{40})\,
\theta(-q^{15};q^{40})\,
\theta(-q;q^{40})
\\
=
-q
\theta(q^{12};q^{40})\,
\theta(q^{2};q^{40})\,
\theta(-q^{25};q^{40})\,
\theta(-q^9;q^{40}),
\end{multline*}
and thus the above right-hand side becomes
$$
\frac {\theta(q^{24};q^{40})\,
\theta(q^{10};q^{40})\,
\theta(-q^{13};q^{40})\,
\theta(-q^{3};q^{40})
(q^{40};q^{40})_\infty}
{\theta(-q;q^{40})\,
\theta(-q^{9};q^{40})\,
\theta(q^{12};q^{40})\,
\theta(-q^{15};q^{40})},
$$
which is equivalent to the right-hand side of \eqref{eq:Id6.2}.
\end{proof}

\begin{proof}[Proof of Theorem \ref{thm:240m+121}]
The point of departure is the first two lines in \eqref{eq:240m+121}.
By \eqref{eq:threl1} with $N=40$ and $u=-q^{37}$, respectively with
$N=40$ and $u=-q^{27}$, we get
$$
\sum_{n=0}^\infty(-1)^{\fl{(n+2)/4}}q^{a_n}
=
\frac
{(q^{34},q^{6},q^{40};q^{40})_\infty}
{(-q^{37},-q^{3};q^{40})_\infty}
+
q
\frac
{(q^{14},q^{26},q^{40};q^{40})_\infty}
{(-q^{27},-q^{13};q^{40})_\infty}.
$$
By \eqref{eq:tadd} with $q$ replaced by $q^{40}$, 
$u=-q^{16}$, $v=q^{11}$, $y=q^{15}$, 
and $x=q^{19}$,
we get
\begin{multline*}
\theta(q^{34};q^{40})\,\theta(q^4;q^{40})\,\theta(-q^{27};q^{40})\,\theta(-q^5;q^{40})-
\theta(q^{30};q^{40})\,\theta(q^8;q^{40})\,\theta(-q^{31};q^{40})\,\theta(-q;q^{40})\\
=-q\,\theta(q^{26};q^{40})\,\theta(q^4;q^{40})\,\theta(-q^{35};q^{40})\,\theta(-q^3;q^{40}),
\end{multline*}
and thus the above right-hand side becomes
$$
\frac
    {\theta(q^{30};q^{40})\,\theta(q^8;q^{40})\,\theta(-q^{31};q^{40})\,\theta(-q;q^{40})\,(q^{40};q^{40})_\infty} 
{\theta(-q^{27};q^{40})\,\theta(-q^{37};q^{40})\,
\theta(q^{4};q^{40})\,\theta(-q^5;q^{40})},
$$
which is equivalent to the right-hand side of \eqref{eq:Id6.5}.
\end{proof}

\begin{proof}[Proof of Theorem \ref{thm:240m+169}]
Our point of departure is the first two lines in \eqref{eq:240m+169}. 
By \eqref{eq:threl1} with $N=40$ and $u=-q^{29}$, respectively 
by \eqref{eq:threl2} with
$N=40$ and $u=-q^{59}$, we get
$$
\sum_{n=0}^\infty(-1)^{\fl{5n/4}}q^{a_n}=
\frac
{(q^{18},q^{22},q^{40};q^{40})_\infty}
{(-q^{29},-q^{11};q^{40})_\infty}
-
q^5
\frac
{(q^{38},q^{2},q^{40};q^{40})_\infty}
{(-q^{19},-q^{21};q^{40})_\infty}.
$$
If we now replace $q$ by $q^{40}$ and choose
$u = -q^{13}$, $v = -q^{15}$, 
$x = -q^{3}$, and $y = q^{8}$
in \eqref{eq:tadd}, we obtain
\begin{multline*}
\theta(-q^{23};q^{40})\,
\theta(q^{10};q^{40})\,
\theta(q^{16};q^{40})\,
\theta(-q^{7};q^{40})\\
-
\theta(-q^{21};q^{40})\,
\theta(q^{12};q^{40})\,
\theta(q^{18};q^{40})\,
\theta(-q^5;q^{40})
\\
=
-q^5
\theta(-q^{11};q^{40})\,
\theta(q^{2};q^{40})\,
\theta(q^{28};q^{40})\,
\theta(-q^5;q^{40})
\end{multline*}
after little manipulation.
Thus, the above right-hand side becomes
$$
\frac {\theta(-q^{23};q^{40})\,
\theta(q^{10};q^{40})\,
\theta(q^{16};q^{40})\,
\theta(-q^{7};q^{40})
(q^{40};q^{40})_\infty}
{\theta(-q^{11};q^{40})\,
\theta(-q^{19};q^{40})\,
\theta(q^{12};q^{40})\,
\theta(-q^5;q^{40})},
$$
which is equivalent to the right-hand side of \eqref{eq:Id6.6}.
\end{proof}

\begin{proof}[Proof of Theorem \ref{thm:15m+1}]
This is a special case of Theorem~\ref{thm:3Pm+a^2}.
\end{proof}

\begin{proof}[Proof of Theorem \ref{thm:15m+4}]
This is a special case of Theorem~\ref{thm:3Pm+a^2}.
\end{proof}

\begin{proof}[Proof of Theorem \ref{thm:6.7}]
As we already said, this is a direct consequence of the Jacobi triple product
identity \eqref{eq:JTP}: one replaces $q$ by $q^{10}$ and then
chooses $z=-q^{6}$ there.
\end{proof}

\begin{proof}[Proof of Theorem \ref{thm:6.8}]
We start again with (cf.\ \eqref{eq:6.8})
$$
\sum_{n=0}^\infty \big(q^{n(n+1)}- q^{5n(n+1)+1}\big)
=
(q^2;q^2)_\infty\,(-q^2;q^2)_\infty^2
-q\,(q^{10};q^{10})_\infty\,(-q^{10};q^{10})_\infty^2.
$$
The right-hand side can be rewritten as
\begin{equation} \label{eq:Aus6.8} 
\frac {(q^4;q^4)_\infty^2}
{(q^2;q^2)_\infty}
-q\frac {(q^{20};q^{20})_\infty^2} {q^{10};q^{10})_\infty}
=
\frac {(q^{20};q^{20})_\infty^2} 
{q^{10};q^{10})_\infty}
\left(
\frac {
\theta(q^4;q^{20})\,
\theta(q^8;q^{20})\,
}
{
\theta(q^2;q^{20})\,
\theta(q^6;q^{20})\,
}
-q\right).
\end{equation}
If in \eqref{eq:tadd} we replace $q$ by $q^{20}$ and choose
$u = q^5$, $v = q^2$, $x = q^{12}$, and $y = q^4$, then we obtain
\begin{multline*}
\theta(q^{16};q^{20})\,\theta(q^8;q^{20})\,
\theta(q^7;q^{20})\,\theta(q^3;q^{20})-
\theta(q^{14};q^{20})\,\theta(q^{10};q^{20})\,
\theta(q^9;q^{20})\,\theta(q;q^{20})\\
=q\,\theta(q^6;q^{20})\,\theta(q^2;q^{20})\,
\theta(q^{17};q^{20})\,\theta(q^7;q^{20}).
\end{multline*}
If this is used in \eqref{eq:Aus6.8}, then we obtain the
right-hand side of \eqref{eq:Id6.8} after little manipulation.
\end{proof}

\begin{proof}[Proof of Theorem \ref{thm:6.9}]
We start again with (cf.\ \eqref{eq:6.9})
\begin{equation} \label{eq:Aus6.9} 
1+\sum_{n=1}^\infty\big( q^{n^2}+q^{5n^2}\big)
=\frac {1} {2}\big((q^2;q^2)_\infty\,(-q;q^2)_\infty^2
+(q^{10};q^{10})_\infty\,(-q^5;q^{10})_\infty^2\big).
\end{equation}
Now, we have the relation\footnote{In combinatorial terms, this is
Euler's theorem that the number of partitions of~$n$ into odd parts
equals the number of partitions of~$n$ into distinct parts.}
 (cf.\ \cite[Eq.~(8.10.9)]{GaRaAA})
\begin{equation} \label{eq:Euler} 
(q;q^2)_\infty^{-1}=(-q;q)_\infty.
\end{equation}
Upon replacement of $q$ by $-q$, we obtain the variant
\begin{equation} \label{eq:Euler2} 
(-q;q^2)_\infty^{-1}=(q;q^2)_\infty\,(-q^2;q^2)_\infty.
\end{equation}
We use the latter identity in \eqref{eq:Aus6.9} and get
\begin{multline*}
1+\sum_{n=1}^\infty\big( q^{n^2}+q^{5n^2}\big)
=\frac {1} {2}\left(
\frac {(q^2;q^2)_\infty\,(-q;q^2)_\infty}
{(q;q^2)_\infty\,(-q^2;q^2)_\infty}
+(q^{10};q^{10})_\infty\,(-q^5;q^{10})_\infty^2\right)\\
=\frac {1} {2}
\left(
\frac {
\theta(q^2;q^{10})\,
\theta(q^4;q^{10})\,
(q^{10};q^{10})_\infty\,
\theta(-q;q^{10})\,
\theta(-q^3;q^{10})\,
(-q^5;q^{10})_\infty
}
{
\theta(q;q^{10})\,
\theta(q^3;q^{10})\,
(q^5;q^{10})_\infty\,
\theta(-q^2;q^{10})\,
\theta(-q^4;q^{10})\,
(-q^{10};q^{10})_\infty
}\right.\\
\left.\vphantom{\frac {\theta(s)} {\theta(s)}}
+(q^{10};q^{10})_\infty\,
(-q^5;q^{10})_\infty^2\right).
\end{multline*}
If, in the first expression within parentheses we use \eqref{eq:Euler2}
with $q$ replaced by $q^5$, then the above becomes
\begin{multline*}
1+\sum_{n=1}^\infty\big( q^{n^2}+q^{5n^2}\big)
=\frac {(q^{10};q^{10})_\infty\,
(-q^5;q^{10})_\infty^2} 
{2\,\theta(q;q^{10})\,
\theta(q^3;q^{10})\,
\theta(-q^2;q^{10})\,
\theta(-q^4;q^{10})
}\\
\times
\left(
\theta(q^2;q^{10})\,
\theta(q^4;q^{10})\,
\theta(-q;q^{10})\,
\theta(-q^3;q^{10})
\right.\\
\left.\vphantom{\frac {\theta(s)} {\theta(s)}}
+
\theta(q;q^{10})\,
\theta(q^3;q^{10})\,
\theta(-q^2;q^{10})\,
\theta(-q^4;q^{10})
\right).
\end{multline*}
We now replace $q$ by $q^{10}$ in \eqref{eq:tadd} and
choose $u = -q^3$, $v = q$, $x = q^4$, and $y = q^3$. This gives the relation
\begin{multline*}
\theta(q^2;q^{10})\,
\theta(q^4;q^{10})\,
\theta(-q;q^{10})\,
\theta(-q^3;q^{10})
+\theta(q;q^{10})\,
\theta(q^3;q^{10})\,
\theta(-q^2;q^{10})\,
\theta(-q^4;q^{10})
\\
=
\theta(q^3;q^{10})\,
\theta(-q^4;q^{10})\,
\theta(q^5;q^{10})\,
\theta(-1;q^{10}).
\end{multline*}
If this is used in the above expression, then we obtain
\eqref{eq:Id6.9} after some manipulation, where \eqref{eq:Euler2} is
again used with $q$ replaced by $q^5$.
\end{proof}

\begin{proof}[Proof of Theorem \ref{thm:6.10}]
As we already said, this is a direct consequence of the Jacobi triple product
identity \eqref{eq:JTP}: one replaces $q$ by $q^{10}$ and then
chooses $z=-q^{7}$ there.
\end{proof}

\begin{proof}[Proof of Theorem \ref{thm:6.11}]
We start again with (cf.\ \eqref{eq:6.11})
$$
  \sum_{n=0}^\infty \big(q^{n(n+1)}+ q^{5n(n+1)+1}\big)
=
(q^2;q^2)_\infty\,(-q^2;q^2)_\infty^2
+q\,(q^{10};q^{10})_\infty\,(-q^{10};q^{10})_\infty^2.
$$
The right-hand side can be rewritten as
\begin{equation} \label{eq:Aus6.11} 
\frac {(q^4;q^4)_\infty^2}
{(q^2;q^2)_\infty}
+q\frac {(q^{20};q^{20})_\infty^2} {(q^{10};q^{10})_\infty}
=
\frac {(q^{20};q^{20})_\infty^2} 
{(q^{10};q^{10})_\infty}
\left(
\frac {
\theta(q^4;q^{20})\,
\theta(q^8;q^{20})\,
}
{
\theta(q^2;q^{20})\,
\theta(q^6;q^{20})\,
}
+q\right).
\end{equation}
If in \eqref{eq:tadd} we replace $q$ by $q^{20}$ and choose
$u = q^5$, $v = q^2$, $x = q^6$, and $y = q^4$, then we obtain
\begin{multline*}
\theta(q^{10};q^{20})\,\theta(q^2;q^{20})\,
\theta(q^7;q^{20})\,\theta(q^3;q^{20})-
\theta(q^{8};q^{20})\,\theta(q^{4};q^{20})\,
\theta(q^9;q^{20})\,\theta(q;q^{20})\\
=q\,\theta(q^6;q^{20})\,\theta(q^2;q^{20})\,
\theta(q^{11};q^{20})\,\theta(q;q^{20}).
\end{multline*}
If this is used in \eqref{eq:Aus6.11}, then we obtain the
right-hand side of \eqref{eq:Id6.11} after little manipulation.
\end{proof}

\begin{proof}[Proof of Theorem \ref{thm:6.12}]
As we already said, this is a direct consequence of the Jacobi triple product
identity \eqref{eq:JTP}: one replaces $q$ by $q^{10}$ and then
chooses $z=-q^{8}$ there.
\end{proof}

\begin{proof}[Proof of Theorem \ref{thm:6.13}]
As we already said, this is a direct consequence of the Jacobi triple product
identity \eqref{eq:JTP}: one replaces $q$ by $q^{10}$ and then
chooses $z=-q^{9}$ there.
\end{proof}

\begin{proof}[Proof of Theorem \ref{thm:6.14}]
We start again with (cf.\ \eqref{eq:6.14})
\begin{equation} \label{eq:Aus6.14} 
  \sum_{n=1}^\infty \big(q^{n^2-1}- q^{5n^2-1}\big)
=\frac {1} {2q}\big((q^2;q^2)_\infty\,(-q;q^2)_\infty^2
-(q^{10};q^{10})_\infty\,(-q^5;q^{10})_\infty^2\big).
\end{equation}
The right-hand side is almost the same expression as the one
on the right-hand side of \eqref{eq:Aus6.9}, except for a 
multiplicative factor of $1/q$ and a changed sign.
Hence, by proceeding as in the alternative proof of 
Theorem~\ref{thm:6.9}, we arrive at
\begin{multline*}
  \sum_{n=1}^\infty \big(q^{n^2-1}- q^{5n^2-1}\big)
=\frac {(q^{10};q^{10})_\infty\,
(-q^5;q^{10})_\infty^2} 
{2q\,\theta(q;q^{10})\,
\theta(q^3;q^{10})\,
\theta(-q^2;q^{10})\,
\theta(-q^4;q^{10})
}\\
\times
\left(
\theta(q^2;q^{10})\,
\theta(q^4;q^{10})\,
\theta(-q;q^{10})\,
\theta(-q^3;q^{10})
\right.\\
\left.\vphantom{\frac {\theta(s)} {\theta(s)}}
-
\theta(q;q^{10})\,
\theta(q^3;q^{10})\,
\theta(-q^2;q^{10})\,
\theta(-q^4;q^{10})
\right).
\end{multline*}
We now replace $q$ by $q^{10}$ in \eqref{eq:tadd} and
choose $u = -q^3$, $v = q$, $x = q^4$, and $y = q^3$. This gives the relation
\begin{multline*}
\theta(q^4;q^{10})\,
\theta(q^2;q^{10})\,
\theta(-q^3;q^{10})\,
\theta(-q;q^{10})
-\theta(-q^4;q^{10})\,
\theta(-q^2;q^{10})\,
\theta(q^3;q^{10})\,
\theta(q;q^{10})
\\
=q\,
\theta(-q^2;q^{10})\,
\theta(-1;q^{10})\,
\theta(q^5;q^{10})\,
\theta(q;q^{10}).
\end{multline*}
If this is used in the above expression, then we obtain
\eqref{eq:Id6.14} after some manipulation, where \eqref{eq:Euler2} is
again used with $q$ replaced by $q^5$.
\end{proof}

\begin{proof}[Proof of Theorem \ref{thm:120m+49}]
This is a special case of Corollary~\ref{cor:24Pm+a^2}.
\end{proof}

\begin{proof}[Proof of Theorem \ref{thm:120m+1}]
This theorem is a special case of Theorem~\ref{thm:24Pm+a^2}.
\end{proof}

\begin{proof}[Proof of Theorem \ref{thm:168m+121}]
This is a special case of Corollary~\ref{cor:24Pm+a^2}.
\end{proof}

\begin{proof}[Proof of Theorem \ref{thm:168m+1}]
This theorem is a special case of Theorem~\ref{thm:24Pm+a^2}.
\end{proof}

\begin{proof}[Proof of Theorem \ref{thm:168m+25}]
This theorem is a special case of Theorem~\ref{thm:24Pm+a^2}.
\end{proof}

\begin{proof}[Proof of Theorem \ref{thm:48m+1}]
Using \eqref{eq:threl1} with $N=8$ and $u=q^7$, the right-hand side
of \eqref{eq:48m+1} can be simplified into one product. 
This gives the claimed result.
\end{proof}

\begin{proof}[Proof of Theorem \ref{thm:48m+25}]
Using \eqref{eq:threl2} with $N=8$ and $u=q^{11}$, the right-hand side
of \eqref{eq:48m+25} can be simplified into one product. 
This gives the claimed result.
\end{proof}

\begin{proof}[Proof of Theorem \ref{thm:21m+1}]
This theorem is a special case of Theorem~\ref{thm:3Pm+a^2}.
\end{proof}

\begin{proof}[Proof of Theorem \ref{thm:21m+4}]
This theorem is a special case of Theorem~\ref{thm:3Pm+a^2}.
\end{proof}

\begin{proof}[Proof of Theorem \ref{thm:21m+16}]
This theorem is a special case of Corollary~\ref{cor:3Pm+a^2}.
\end{proof}

\begin{proof}[Proof of Theorem \ref{thm:16m+1}]
This is a special case of Theorem~\ref{thm:16m+a^2}.
\end{proof}

\begin{proof}[Proof of Theorem \ref{thm:16m+9}]
This is a special case of Theorem~\ref{thm:16m+a^2}.
\end{proof}

\section{Parametric families of generating functions for 
sequences of squares}
\label{sec:Ids3}

Here we present two results on generating functions for sequences of
squares that contain parameters. The first of these results consists
in Theorem~\ref{thm:24Pm+a^2} and Corollary~\ref{cor:24Pm+a^2}, while
the second consists 
in Theorem~\ref{thm:3Pm+a^2} and Corollary~\ref{cor:3Pm+a^2}.
For each theorem-corollary pair the proof is the same, but the theorem
covers a parameter range that is different from that of the corollary.
We conclude the section by stating, and proving, a uniform version of
Theorems~\ref{thm:16m+1} and~\ref{thm:16m+9}.

\medskip
The following theorem covers Theorems~\ref{thm:120m+1},
\ref{thm:168m+1} and \ref{thm:168m+25}.

\begin{theorem} \label{thm:24Pm+a^2}
Let\/ $P$ be an odd prime power different from a power of~$3$, 
and let $a$ be an odd 
positive integer relatively prime to~$P$ and less than~$P$.
	Let\/ $(a_n)_{n\ge0}$ be the sequence of non-negative integers $m$ such
	that $24Pm+a^2$ is a square. 
\begin{enumerate}
\item If \hbox{$a\equiv P$}~{\em(mod~$3$)}, then
	\begin{equation} \label{eq:24Pm+a^2-1}
	\sum_{n=0}^\infty (-1)^{\fl{(n+2)/4}}q^{a_n}=
	\frac {(q^{(P-a)/3},q^{(2P+a)/3},q^P;q^P)_\infty} 
{(q^{(P-a)/6},q^{(5P+a)/6};q^P)_\infty}.
	\end{equation}
\item If \hbox{$a\not\equiv P$}~{\em(mod~$3$)}, then
	\begin{equation} \label{eq:24Pm+a^2-2}
	\sum_{n=0}^\infty (-1)^{\fl{(n+2)/4}}q^{a_n}=
	\frac {(q^{(P+a)/3},q^{(2P-a)/3},q^P;q^P)_\infty} 
{(q^{(P+a)/6},q^{(5P-a)/6};q^P)_\infty}.
	\end{equation}
\end{enumerate}
\end{theorem}

\begin{proof}
We have to find all $S$ such that 
\begin{equation} \label{eq:S^2a^2} 
S^2\equiv a^2\pmod{24P}.
\end{equation}
We claim that there are the following two cases:

\begin{enumerate}
\item[(C1)] If $a\equiv P$~(mod~$3$), then
$S\equiv a,2P-a,4P+a,6P-a$~$(\text{mod}~6P)$.
\item[(C2)] If $a\not\equiv P$~(mod~$3$), then
$S\equiv a,2P+a,4P-a,6P-a$~$(\text{mod}~6P)$.
\end{enumerate}

It should be noted that the condition $a<P$ guarantees that,
in both cases, the
congruence classes above are listed in increasing order.

By assumption, $a$ is odd. Hence, also $S$ must be odd
and automatically $S^2\equiv a^2\equiv1$~(mod~$8$).
Consequently, the congruence \eqref{eq:S^2a^2} can be reduced to
$$
(S-a)(S+a)\equiv0\pmod{6P}.
$$
Two solutions are immediate, namely $S\equiv a$~(mod~$6P$) and
$S\equiv -a$~(mod~$6P$). There are two further possibilities:
\begin{equation} \label{eq:N1} 
S\equiv a~(\text{mod }6)\quad \text{and}\quad 
S\equiv -a~(\text{mod }P),
\end{equation}
and
\begin{equation} \label{eq:N2} 
S\equiv -a~(\text{mod }6)\quad \text{and}\quad 
S\equiv a~(\text{mod }P).
\end{equation}
Under our assumption that $a$ is relatively prime to~$P$, there
are no other possibilities. For, writing $P=p^e$ with $p$ an odd
prime number, the simultaneous congruences
$$
S\equiv a~(\text{mod }p^\alpha)\quad \text{and}\quad 
S\equiv -a~(\text{mod }p^\beta)
$$
with $\alpha+\beta=e$ and both $\alpha$ and $\beta$ positive
would imply that $2a$ is divisible by~$p$, a contradiction to
our assumptions. 

It is straightforward to see that, depending on whether $a\equiv
P$~(mod~$3$) or not, the solutions to \eqref{eq:N1} respectively to
\eqref{eq:N2} are given by the second and third option in (C1) and (C2)
above. 

\medskip
We now discuss Case~(C1). Here, we have
\begin{multline*} 
\sum_{n=0}^\infty(-1)^{\fl{(n+2)/4}}q^{a_n}=
\sum_{k=0}^\infty (-1)^kq^{\frac {1} {24P}((6Pk+a)^2-a^2)}
+\sum_{k=0}^\infty (-1)^kq^{\frac {1} {24P}((6Pk+2P-a)^2-a^2)}\\
-\sum_{k=0}^\infty (-1)^kq^{\frac {1} {24P}((6Pk+4P+a)^2-a^2)}
-\sum_{k=0}^\infty (-1)^kq^{\frac {1} {24P}((6Pk+6P-a)^2-a^2)}\\
=
\sum_{k=0}^\infty (-1)^kq^{\frac 12(3Pk^2+ak)}
+\sum_{k=0}^\infty (-1)^kq^{\frac 12(3Pk^2+(2P-a)k)+\frac {1} {6}(P-a)}
\kern3cm\\
-\sum_{k=0}^\infty (-1)^kq^{\frac 12(3Pk^2+(4P+a)k)+\frac {1} {3}(2P+a)}
-\sum_{k=0}^\infty (-1)^kq^{\frac 12(3Pk^2+(6P-a)k)+\frac {1} {2}(3P-a)}.
\end{multline*}
Again, sums can be put together in pairs, 
so that one obtains two sums over {\it all\/} integers~$k$:
$$
\sum_{n=0}^\infty(-1)^{\fl{(n+2)/4}}q^{a_n}=
\sum_{k=-\infty}^\infty (-1)^kq^{\frac 12(3Pk^2+ak)}
+\sum_{k=-\infty}^\infty (-1)^kq^{\frac 12(3Pk^2+(2P-a)k)+\frac {1} {6}(P-a)}.
$$
Now, to each of these sums we apply the Jacobi triple product identity
\eqref{eq:JTP} to get
\begin{multline*}
\sum_{n=0}^\infty(-1)^{\fl{(n+2)/4}}q^{a_n}=
(q^{3P},q^{(3P+a)/2},q^{(3P-a)/2};q^{3P})_\infty\\
+q^{(P-a)/6}\,(q^{3P},q^{(5P-a)/2},q^{(P+a)/2};q^{3P})_\infty.
\end{multline*}
The sum of these two products simplifies to the single product on
the right-hand side of \eqref{eq:24Pm+a^2-1} as is
seen by applying \eqref{eq:threl1} with $N=P$ and $u=q^{(5P+a)/6}$.

\medskip
On the other hand, if we are in Case~(C2), then we have
\begin{multline*} 
\sum_{n=0}^\infty(-1)^{\fl{(n+2)/4}}q^{a_n}=
\sum_{k=0}^\infty (-1)^kq^{\frac {1} {24P}((6Pk+a)^2-a^2)}
+\sum_{k=0}^\infty (-1)^kq^{\frac {1} {24P}((6Pk+2P+a)^2-a^2)}\\
-\sum_{k=0}^\infty (-1)^kq^{\frac {1} {24P}((6Pk+4P-a)^2-a^2)}
-\sum_{k=0}^\infty (-1)^kq^{\frac {1} {24P}((6Pk+6P-a)^2-a^2)}\\
=
\sum_{k=0}^\infty (-1)^kq^{\frac 12(3Pk^2+ak)}
+\sum_{k=0}^\infty (-1)^kq^{\frac 12(3Pk^2+(2P+a)k)+\frac {1} {6}(P+a)}
\kern3cm\\
-\sum_{k=0}^\infty (-1)^kq^{\frac 12(3Pk^2+(4P-a)k)+\frac {1} {3}(2P-a)}
-\sum_{k=0}^\infty (-1)^kq^{\frac 12(3Pk^2+(6P-a)k)+\frac {1} {2}(3P-a)}.
\end{multline*}
Again, sums can be put together in pairs, 
so that one obtains two sums over {\it all\/} integers~$k$:
$$
\sum_{n=0}^\infty(-1)^{\fl{(n+2)/4}}q^{a_n}=
\sum_{k=-\infty}^\infty (-1)^kq^{\frac 12(3Pk^2+ak)}
+\sum_{k=-\infty}^\infty (-1)^kq^{\frac 12(3Pk^2+(2P+a)k)+\frac {1} {6}(P+a)}.
$$
Now, to each of these sums we apply the Jacobi triple product identity
\eqref{eq:JTP} to get
\begin{multline*}
\sum_{n=0}^\infty(-1)^{\fl{(n+2)/4}}q^{a_n}=
(q^{3P},q^{(3P+a)/2},q^{(3P-a)/2};q^{3P})_\infty\\
+q^{(P+a)/6}\,(q^{3P},q^{(5P+a)/2},q^{(P-a)/2};q^{3P})_\infty.
\end{multline*}
The sum of these two products simplifies to the single product on
the right-hand side of \eqref{eq:24Pm+a^2-2} as is
seen by applying \eqref{eq:threl1} with $N=P$ and $u=q^{(5P-a)/6}$.
\end{proof}

The following corollary covers Theorems~\ref{thm:120m+49} and
\ref{thm:168m+121}.

\begin{corollary} \label{cor:24Pm+a^2}
Let\/ $P$ be an odd prime power different from a power of~$3$, 
and let $a$ be an odd 
positive integer relatively prime to~$P$ with $P<a<2P$.
	Let\/ $(a_n)_{n\ge0}$ be the sequence of non-negative integers $m$ such
	that $24Pm+a^2$ is a square. 
\begin{enumerate}
\item If \hbox{$a\equiv P$}~{\em(mod~$3$)}, then
	\begin{equation} 
	\sum_{n=0}^\infty (-1)^{\fl{(n+2)/4}}q^{a_n}=
	\frac {(q^{(P-a)/3},q^{(2P+a)/3},q^P;q^P)_\infty} 
{(q^{(P-a)/6},q^{(5P+a)/6};q^P)_\infty}.
	\end{equation}
\item If \hbox{$a\not\equiv P$}~{\em(mod~$3$)}, then
	\begin{equation} 
	\sum_{n=0}^\infty (-1)^{\fl{5n/4}}q^{a_n}=
	\frac {(q^{(P+a)/3},q^{(2P-a)/3},q^P;q^P)_\infty} 
{(q^{(P+a)/6},q^{(5P-a)/6};q^P)_\infty}.
	\end{equation}
\end{enumerate}
\end{corollary}

\begin{proof}
If one looks through the arguments of the proof of
Theorem~\ref{thm:24Pm+a^2}, then one sees that everything can be
copied verbatim, until it comes to the description of
the two cases to be considered: we have to adapt the order of
the congruence classes, as shown below.

\begin{enumerate}
\item[(C1')] If $a\equiv P$~(mod~$3$), then
$S\equiv 2P-a,a,6P-a,4P+a$~$(\text{mod}~6P)$.
\item[(C2')] If $a\not\equiv P$~(mod~$3$), then
$S\equiv a,4P-a,2P+a,6P-a$~$(\text{mod}~6P)$.
\end{enumerate}

While in Case~(C1'), the rest of the proof can be copied, in Case~(C2')
this requires a change in sign in the original sum from
$(-1)^{\fl{(n+2)/4}}$ to $(-1)^{\fl{5n/4}}$.
\end{proof}

The following theorem covers Theorems~\ref{thm:15m+1}, \ref{thm:15m+4},
\ref{thm:21m+1} and \ref{thm:21m+4}.

\begin{theorem} \label{thm:3Pm+a^2}
Let\/ $P$ be an odd prime power different from a power of~$3$,  
and let $a$ be a
positive integer relatively prime to~$P$ and less than $P/2$.
	Let\/ $(a_n)_{n\ge0}$ be the sequence of non-negative integers $m$ such
	that $3Pm+a^2$ is a square. 
\begin{enumerate}
\item If \hbox{$a\equiv P$}~{\em(mod~$3$)}, then
	\begin{equation} \label{eq:3Pm+a^2-1}
	\sum_{n=0}^\infty (-1)^{\fl{(n+2)/4}}q^{a_n}=
\frac
{(q^{(4P-4a)/3},q^{(2P+4a)/3},q^{2P};q^{2P})_\infty}
{(q^{(5P-2a)/3},q^{(P+2a)/3};q^{2P})_\infty}.
	\end{equation}
\item If \hbox{$a\not\equiv P$}~{\em(mod~$3$)}, then
	\begin{equation} 
	\sum_{n=0}^\infty (-1)^{\fl{(n+2)/4}}q^{a_n}=
\frac
{(q^{(4P+4a)/3},q^{(2P-4a)/3},q^{2P};q^{2P})_\infty}
{(q^{(5P+2a)/3},q^{(P-2a)/3};q^{2P})_\infty}.
	\end{equation}
\end{enumerate}
\end{theorem}

\begin{proof}
We have to find all $S$ such that 
\begin{equation*} 
S^2\equiv a^2\pmod{3P}.
\end{equation*}
As before, there are two cases:

\begin{enumerate}
\item[(C1)] If $a\equiv P$~(mod~$3$), then
$S\equiv a,P+a,2P-a,3P-a$~$(\text{mod}~3P)$.
\item[(C2)] If $a\not\equiv P$~(mod~$3$), then
$S\equiv a,P-a,2P+a,3P-a$~$(\text{mod}~3P)$.
\end{enumerate}

\medskip
We now discuss Case~(C1). Here, we have
\begin{multline*} 
\sum_{n=0}^\infty(-1)^{\fl{(n+2)/4}}q^{a_n}=
\sum_{k=0}^\infty (-1)^kq^{\frac {1} {3P}((3Pk+a)^2-a^2)}
+\sum_{k=0}^\infty (-1)^kq^{\frac {1} {3P}((3Pk+P+a)^2-a^2)}\\
-\sum_{k=0}^\infty (-1)^kq^{\frac {1} {3P}((3Pk+2P-a)^2-a^2)}
-\sum_{k=0}^\infty (-1)^kq^{\frac {1} {3P}((3Pk+3P-a)^2-a^2)}\\
=
\sum_{k=0}^\infty (-1)^kq^{3Pk^2+2ak}
+\sum_{k=0}^\infty (-1)^kq^{3Pk^2+2(P+a)k+\frac {1} {3}(P+2a)}
\kern3cm\\
-\sum_{k=0}^\infty (-1)^kq^{3Pk^2+2(2P-a)k+\frac {4} {3}(P-a)}
-\sum_{k=0}^\infty (-1)^kq^{3Pk^2+2(3P-a)k+3P-2a}.
\end{multline*}
Again, sums can be put together in pairs, 
so that one obtains two sums over {\it all\/} integers~$k$:
$$
\sum_{n=0}^\infty(-1)^{\fl{(n+2)/4}}q^{a_n}=
\sum_{k=-\infty}^\infty (-1)^kq^{3Pk^2+2ak}
+\sum_{k=-\infty}^\infty (-1)^kq^{3Pk^2+2(P+a)k+\frac {1} {3}(P+2a)}.
$$
Now, to each of these sums we apply the Jacobi triple product identity
\eqref{eq:JTP} to get
\begin{multline*}
\sum_{n=0}^\infty(-1)^{\fl{(n+2)/4}}q^{a_n}=
(q^{6P},q^{3P+2a},q^{3P-2a};q^{6P})_\infty\\
+q^{(P+2a)/3}\,(q^{6P},q^{5P+2a},q^{P-2a};q^{6P})_\infty.
\end{multline*}
The sum of these two products simplifies to the single product on
the right-hand side of \eqref{eq:3Pm+a^2-1} as is
seen by applying \eqref{eq:threl1} with $N=2P$ and $u=q^{(5P-2a)/3}$.

\medskip
On the other hand, if we are in Case~(C2), then we have
\begin{multline*} 
\sum_{n=0}^\infty(-1)^{\fl{(n+2)/4}}q^{a_n}=
\sum_{k=0}^\infty (-1)^kq^{\frac {1} {3P}((3Pk+a)^2-a^2)}
+\sum_{k=0}^\infty (-1)^kq^{\frac {1} {3P}((3Pk+P-a)^2-a^2)}\\
-\sum_{k=0}^\infty (-1)^kq^{\frac {1} {3P}((3Pk+2P+a)^2-a^2)}
-\sum_{k=0}^\infty (-1)^kq^{\frac {1} {3P}((3Pk+3P-a)^2-a^2)}\\
=
\sum_{k=0}^\infty (-1)^kq^{3Pk^2+2ak}
+\sum_{k=0}^\infty (-1)^kq^{3Pk^2+2(P-a)k+\frac {1} {3}(P-2a)}
\kern3cm\\
-\sum_{k=0}^\infty (-1)^kq^{3Pk^2+2(2P+a)k+\frac {4} {3}(P+a)}
-\sum_{k=0}^\infty (-1)^kq^{3Pk^2+2(3P-a)k+3P-2a}.
\end{multline*}
Again, sums can be put together in pairs, 
so that one obtains two sums over {\it all\/} integers~$k$:
$$
\sum_{n=0}^\infty(-1)^{\fl{(n+2)/4}}q^{a_n}=
\sum_{k=-\infty}^\infty (-1)^kq^{3Pk^2+2ak}
+\sum_{k=-\infty}^\infty (-1)^kq^{3Pk^2+2(P-a)k+\frac {1} {3}(P-2a)}.
$$
Now, to each of these sums we apply the Jacobi triple product identity
\eqref{eq:JTP} to get
\begin{multline*}
\sum_{n=0}^\infty(-1)^{\fl{(n+2)/4}}q^{a_n}=
(q^{6P},q^{3P+2a},q^{3P-2a};q^{6P})_\infty\\
+q^{(P-2a)/3}\,(q^{6P},q^{5P-2a},q^{P+2a};q^{6P})_\infty.
\end{multline*}
The sum of these two products simplifies to the single product on
the right-hand side of \eqref{eq:3Pm+a^2-1} as is
seen by applying \eqref{eq:threl1} with $N=2P$ and $u=q^{(5P+2a)/3}$.
\end{proof}

The following corollary covers Theorem~\ref{thm:21m+16}.

\begin{corollary} \label{cor:3Pm+a^2}
Let\/ $P$ be an odd prime power different from a power of~$3$,  
and let $a$ be a
positive integer relatively prime to~$P$ with $P/2<a<P$.
	Let\/ $(a_n)_{n\ge0}$ be the sequence of non-negative integers $m$ such
	that $3Pm+a^2$ is a square. 
\begin{enumerate}
\item If \hbox{$a\equiv P$}~{\em(mod~$3$)}, then
	\begin{equation} 
	\sum_{n=0}^\infty (-1)^{\fl{5n/4}}q^{a_n}=
\frac
{(q^{(4P-4a)/3},q^{(2P+4a)/3},q^{2P};q^{2P})_\infty}
{(q^{(5P-2a)/3},q^{(P+2a)/3};q^{2P})_\infty}.
	\end{equation}
\item If \hbox{$a\not\equiv P$}~{\em(mod~$3$)}, then
	\begin{equation} 
	\sum_{n=0}^\infty (-1)^{\fl{(n+2)/4}}q^{a_n}=
\frac
{(q^{(4P+4a)/3},q^{(2P-4a)/3},q^{2P};q^{2P})_\infty}
{(q^{(5P+2a)/3},q^{(P-2a)/3};q^{2P})_\infty}.
	\end{equation}
\end{enumerate}
\end{corollary}

\begin{proof}
Here, one goes through the proof of Theorem~\ref{thm:3Pm+a^2}.
Everything can be copied, except that the description of the
two cases to be considered now reads as shown below, and that
this requires a modification of the sign in Case~(C1').

\begin{enumerate}
\item[(C1')] If $a\equiv P$~(mod~$3$), then
$S\equiv a,2P-a,P+a,3P-a$~$(\text{mod}~3P)$.
\item[(C2')] If $a\not\equiv P$~(mod~$3$), then
$S\equiv P-a,a,3P-a,2P+a$~$(\text{mod}~3P)$.\qedhere
\end{enumerate}
\end{proof}

The following theorem covers Theorems~\ref{thm:16m+1} and \ref{thm:16m+9}.

\begin{theorem} \label{thm:16m+a^2}
Let\/ $a$ be $1$ or $3$.
Furthermore, 
let $(a_n)_{n\ge0}$ be the sequence of non-negative integers $m$ such
	that $16m+a^2$ is a square. Then
	\begin{equation} \label{eq:11a^2}
	\sum_{n=0}^\infty q^{a_n}=
(q^8;q^8)_\infty\,
(-q^{4+a};q^8)_\infty\,
(-q^{4-a};q^8)_\infty.
	\end{equation}
\end{theorem}

\begin{proof}
We have to find all $S$ such that 
$$
S^2\equiv a^2\pmod{16},
$$
or, equivalently,
\begin{equation*} 
(S-a)(S+a)\equiv 0\pmod{16}.
\end{equation*}
Since $a$ is odd, only one of the factors $N-a$ and $N+a$ can be
divisible by~$4$. Hence, either
$S\equiv a$~$(\text{mod}~8)$ or
$S\equiv -a$~$(\text{mod}~8)$.
Consequently, we have
\begin{align*} 
\sum_{n=0}^\infty q^{a_n}&=
\sum_{k=0}^\infty q^{\frac {1} {16}((8k+a)^2-a^2)}
+\sum_{k=1}^\infty q^{\frac {1} {16}((8k-a)^2-a^2)}\\
&=
\sum_{k=0}^\infty q^{4k^2+ak}
+\sum_{k=1}^\infty q^{4k^2-ak}
=
\sum_{k=-\infty}^\infty q^{4k^2+ak}.
\end{align*} 
The proof is completed by applying the Jacobi triple product
identity \eqref{eq:JTP}.
\end{proof}

\section{Consequences and open problems}
\label{sec:open}

In this section, we record a consequence of Theorem~\ref{thm:840m+361}
that is inspired by earlier work of Andrews and the second author 
\cite{Andrews12}. Furthermore, we end by reminding the reader of
a conjecture from \cite{MercAA} related to
Theorems~\ref{thm:840m+361}--\ref{thm:840m+169}.

In \cite{Andrews12}, the function $M_k(n)$ is defined 
as the number of partitions of $n$ in which $k$ is
the least positive integer that is not a part and there are more parts
$>k$ than there are parts $<k$. For example,  
if $n=18$ and $k=3$ then we have $M_3(18)=3$ because the three
partitions in question are
$$5+5+5+2+1,\quad 6+5+4+2+1,\quad \text{and} \quad 7+4+4+2+1.$$

Let $A(n)$ be the number of the partitions of $n$ into parts not
congruent to $0$, $7$, $8$, $13$, $15$, $20$, $22$, $27$, $28 \pmod
{35}$ and the parts congruent to $4$, $9$, $11$, $16$, $19$, $24$,
$26$, $31 \pmod {35}$ have two colours. 

We have the following corollary of Theorem~\ref{thm:840m+361}.

\begin{corollary}
Let\/ $k$ and $n$ be positive integers. With $a_n$ and $t(n)$ as
in Theorem~\ref{thm:840m+361}, we have
$$
(-1)^{k-1}\left( \sum_{j=-(k-1)}^{k} (-1)^j
        A\big(n-j(3j-1)/2\big)-\delta(n)\right) = \sum_{j=0}^n
        (-1)^{t(j)} M_k(n-a_j), 
$$
where
$$
\delta(n) =
\begin{cases}
(-1)^{t(m)}, &\text{if $n=a_m$,}\\ 
0, &\text{otherwise.}
\end{cases} 
$$
\end{corollary}

\begin{proof}
Elementary generating function calculus gives
\begin{equation}\label{gf:A}
\sum_{n=0}^\infty A(n) q^n = \frac {1}
                    {(q,q^4;q^5)_\infty \,(q^2,q^3,q^4,q^5;q^7)_\infty} .
\end{equation}
Similarly, it is not difficult to see that the generating function
for the numbers $M_k(n)$ is given by
\begin{equation} \label{eq:Mkn} 
\sum_{n=0}^\infty M_k(n) q^n = \sum_{n=k}^\infty \frac{q^{{\binom k
      2}+(k+1)n}}{(q;q)_n} 
\begin{bmatrix}
n-1\\k-1
\end{bmatrix},
\end{equation}
where
$$
\begin{bmatrix}
n\\k
\end{bmatrix} 
=
\begin{cases}
\dfrac {(1-q)(1-q^2)\cdots(1-q^n)} 
{(1-q)(1-q^2)\cdots(1-q^k)\,(1-q)(1-q^2)\cdots(1-q^{n-k})}, &  \text{if
                  $0\leqslant k\leqslant n$},\\ 
0, &\text{otherwise,}
\end{cases}
$$
is the usual $q$-binomial coefficient.


Andrews and the second author \cite{Andrews12}
proved the following truncated form of Euler's pentagonal number
theorem~\eqref{eq:pent}:
\begin{equation} \label{TPNT}
\frac{(-1)^{k-1}}{(q;q)_\infty} \sum_{n=-(k-1)}^{k}
                (-1)^{n} q^{n(3n-1)/2}= (-1)^{k-1}+ \sum_{n=k}^\infty
                \frac{q^{{\binom k 2}+(k+1)n}}{(q;q)_n} 
\begin{bmatrix}
n-1\\k-1
\end{bmatrix}.
\end{equation}

Multiplying both sides of \eqref{TPNT} by \eqref{gf:A}, and using
Theorem~\ref{thm:840m+361} and \eqref{eq:Mkn}, we obtain 
\begin{multline*}
 (-1)^{k-1} \left( \bigg( \sum_{n=1}^\infty A(n) q^n \bigg) \bigg(
  \sum_{n=-(k-1)}^{k} (-1)^{n} q^{n(3n-1)/2}\bigg) -\sum_{n=0}^\infty
  (-1)^{t(n)} q^{a_n}\right)   \\ 
 =   \left( \sum_{n=0}^\infty (-1)^{t(n)} q^{a_n} \right) \left(
  \sum_{n=0}^\infty M_k(n) q^n\right).
\end{multline*}
The assertion of the corollary now follows by comparing coefficients
of $q^n$ on both sides of this equation.
\end{proof}

According to \eqref{TPNT}, for $k>0$, the coefficients of $q^n$ in the series
$$
(-1)^{k-1} \left( \frac{1}{(q;q)_\infty} \sum_{j=-(k-1)}^{k} (-1)^{j}
q^{j(3j-1)/2}-1\right)  
$$
are all zero for $0\leqslant n < k(3k+1)/2$, and for  $n\geqslant
k(3k+1)/2$ all the coefficients are positive. 
Related to this result on truncated pentagonal number series, 
we remark that there is substantial numerical evidence
that there is in fact a stronger result. 

\begin{conjecture}\label{C41}
For $k>0$, the coefficients of $q^n$ in the series
\begin{multline*}
(-1)^{k-1}  \left( \frac{1}{(q;q)_\infty} \sum_{j=-(k-1)}^{k} (-1)^{j}
q^{j(3j-1)/2}-1\right) (q,q^6,q^{7};q^{7})_\infty \\
=
\frac{(-1)^k}{(q^2,q^3,q^4,q^5;q^7)_\infty}
\sum_{j=k} ^\infty (-1)^{j}q^{j(3j+1)/2}\left(1-q^{2j+1}\right)
\end{multline*}
are all zero for $0\leqslant n < k(3k+1)/2$ and $n=k(3k+1)/2+1$. For
$n=k(3k+1)/2$ and   $n\geqslant k(3k+1)/2+2$ all the coefficients are
positive.  
\end{conjecture}

\begin{remark}
The equality above follows easily from \eqref{eq:pent} and little
manipulation.
\end{remark}

If we assume Conjecture~\ref{C41}, then
we immediately deduce that the
partition function $A(n)$ satisfies the following infinite families
of linear inequalities. 

\begin{conjecture}
For $k>0$, we have
$$
(-1)^{k-1} \left( \sum_{j=-(k-1)}^k (-1)^j A\big(n-j(3j-1)/2\big)
-\delta(n) \right) \geqslant 0, 
$$
with strict inequalities if $n=k(3k+1)/2$ or $n\geqslant k(3k+1)/2+2$.
\end{conjecture}

To conclude the article, we want to recall a conjecture from
\cite{MercAA} that is very similar in appearance to
Conjecture~\ref{C41} and is related (again via \eqref{eq:pent}) to
Theorems~\ref{thm:840m+361}--\ref{thm:840m+169}.

\begin{conjecture}\label{Cexp}
For $k>0$ and $S\in\{1,2,3,4,5,6\}$, 
the coefficients of $q^n$ in the series
$$
\frac {(-1)^k} {(q,q^4;q^5)_\infty}
\sum_{j=k}^\infty (-1)^{j}
q^{7j(j+1)/2-jS}\left(1-q^{(2j+1)S}\right)
$$
and
$$
\frac {(-1)^k} {(q^2,q^3;q^5)_\infty}
\sum_{j=k}^\infty (-1)^{j}
q^{7j(j+1)/2-jS}\left(1-q^{(2j+1)S}\right)
$$
are non-negative.
\end{conjecture}

\section*{Acknowledgement}
We thank the anonymous referees for a very careful reading of the
original manuscript.

\end{document}